\documentclass[11pt]{amsart}
\usepackage{amsmath,amssymb,mathtools}
\usepackage{microtype}
\usepackage[margin=1.08in]{geometry}
\usepackage{enumitem}
\usepackage{booktabs,longtable,array}
\usepackage{xcolor}
\usepackage{tikz}
\usetikzlibrary{arrows.meta,decorations.pathreplacing,patterns}
\usepackage{caption}
\definecolor{figblue}{RGB}{29,79,118}
\definecolor{figorange}{RGB}{170,82,22}
\definecolor{figteal}{RGB}{17,111,96}
\usepackage[colorlinks=true,linkcolor=blue!55!black,citecolor=blue!55!black,urlcolor=blue!55!black]{hyperref}
\usepackage{fancyhdr}

\numberwithin{equation}{section}
\newtheorem{theorem}{Theorem}[section]

\newtheorem{lemma}[theorem]{Lemma}
\newtheorem{corollary}[theorem]{Corollary}

\newtheorem{remark}[theorem]{Remark}

\DeclareMathOperator{\Sc}{Sc}
\DeclareMathOperator{\Ric}{Ric}
\DeclareMathOperator{\rad}{Rad}
\DeclareMathOperator{\Hess}{Hess}
\DeclareMathOperator{\tr}{tr}
\DeclareMathOperator{\diver}{div}
\DeclareMathOperator{\Lip}{Lip}
\DeclareMathOperator{\Vol}{Vol}

\DeclareMathOperator{\supp}{supp}
\DeclareMathOperator{\arcosh}{arcosh}
\newcommand{\R}{\mathbb R}
\newcommand{\HH}{\mathbb H}
\newcommand{\SSS}{\mathbb S}

\newcommand{\Q}{\mathcal Q}

\newcommand{\cF}{\mathcal F}
\newcommand{\op}{\mathrm{op}}
\newcommand{\tn}{\mathrm{tr}}
\newcommand{\eps}{\varepsilon}

\author{Xumin Jiang}
\address{Xumin Jiang, School of Sciences, Great Bay University, Dongguan 523000, China}
\email{xjiang@gbu.edu.cn}

\author{Zhehui Wang}
\address{Zhehui Wang, School of Sciences, Great Bay University, Dongguan 523000, China}
\email{wangzhehui@gbu.edu.cn}

\title{Hyperbolic comparison and rigidity for the hyperspherical radius}

\date{}

\begin{document}

\begin{abstract}
We prove Gromov's hyperbolic hyperspherical radius comparison and rigidity conjectures in all dimensions $n\ge3$, without the spin assumption. Let $(X^n, g)$ be a smooth, oriented, connected, compact Riemannian manifold with nonempty smooth connected boundary $S$. If $\Sc_g\ge-n(n-1)$ and $H_S\ge(n-1)\coth R$ for some $R>0$, then $\rad(S,g_S)\le\sinh R$, with equality if and only if $(X^n,g)$ is isometric to the closed hyperbolic geodesic ball of radius $R$. Under the same curvature bounds, every smooth $1$-Lipschitz map of nonzero degree from $S$ to the boundary of the hyperbolic geodesic ball of radius $R$ extends to an isometry of the fillings. Here, $g_S$ is the induced boundary metric, $H_S$ is the mean curvature of $S$ with respect to the outward unit normal, and $\rad(S,g_S)$ denotes the hyperspherical radius.
\end{abstract}

\maketitle

\section{Introduction}
The scalar curvature comparison for manifolds with boundary asks how a lower bound of scalar curvature constrains the geometry of the boundary. A central problem is to obtain a sharp upper bound for the intrinsic size of the boundary from a lower bound for its mean
curvature, and to characterize the geometry of the filling when equality holds.
Gromov formulated such a comparison using the hyperspherical radius \cite{GromovLectures, GromovBoundary}, which measures the size of boundary through maps of nonzero degree to round spheres.

For a closed oriented Riemannian $m$-manifold $(S,g_S)$,
its {\it hyperspherical radius} is
\begin{equation}\label{eq:radius}
 \rad(S,g_S)=\sup\{r>0:\exists \text{ smooth map }F:S\longrightarrow
 \SSS^m(r),\ \deg F\ne0,\ \Lip F\le1\}.
\end{equation}
where $\SSS^m(r)$ carries its round intrinsic metric. This
definition does not require $S$ to be diffeomorphic to a sphere.

Let $(X^n,g)$ be a compact, connected, oriented, smooth Riemannian
manifold with nonempty smooth connected boundary $S=\partial X$.
Let $g_S$ be the induced boundary metric, and $H_S$ be the mean curvature of $S$ with respect to the outward unit normal. Gromov \cite[Section 4.3]{GromovLectures} (see also \cite[Section 1]{GromovBoundary}) conjectured the sharp estimate
\begin{equation}\label{eq:gromov-hyperbolic}
 \inf_S H_S
 \le
 (n-1)\sqrt{1+\rad(S,g_S)^{-2}}
 \qquad
 \text{if }\Sc_g\ge-n(n-1).
\end{equation}
The model is a geodesic ball
of radius $R$ in hyperbolic space of sectional curvature $-1$.
Its boundary is isometric to $\SSS^{n-1}(\sinh R)$ and has mean
curvature $(n-1)\coth R$, so equality holds in
\eqref{eq:gromov-hyperbolic}. Gromov also explicitly formulated this conjecture as a comparison of hyperspherical radii.
Equivalently, under the scalar curvature bound $\Sc_g\ge-n(n-1)$, the conjecture asserts that, for every $R>0$, if 
$H_S\ge(n-1)\coth R$ on $S$, then $\operatorname{Rad}(S,g_S)\le\sinh R$.

Our main theorem establishes this comparison and its equality case in every dimension $n\ge3$, without the spin assumption. We state the result as a comparison of hyperspherical radii.

\begin{theorem}\label{thm:hyperbolic}
Let $(X^n,g)$ be a smooth, oriented, connected, compact Riemannian
manifold of dimension $n\ge3$, with nonempty smooth connected
boundary $S=\partial X$. Let $g_S$ be the induced boundary metric, and $H_S$ be the mean curvature of $S$ with respect to the outward unit normal. Suppose that, for some constant $R>0$,
\begin{equation}\label{eq:thm:main-1}
 \Sc_g\ge-n(n-1)\quad\text{in }X,
 \qquad
 H_S\ge(n-1)\coth R\quad\text{on }S.
\end{equation}
Then
\begin{equation}\label{eq:thm:main-2}
 \rad(S,g_S)\le\sinh R.
\end{equation}
Equality holds if and only if $(X,g)$ is isometric to the closed
hyperbolic ball $\overline{B_{\HH^n}(R)}$. 
\end{theorem}

Theorem \ref{thm:hyperbolic} also yields the sharp mean curvature estimate \eqref{eq:gromov-hyperbolic}, with equality if and only if $(X,g)$ is isometric to a hyperbolic geodesic ball.

Gromov \cite[Section 3.5]{GromovLectures} also proposed the corresponding rigidity conjecture for smooth $1$-Lipschitz maps of nonzero degree to the boundaries of hyperbolic balls, with the boundary mean curvature bounded below by that of
the target sphere. The equality case of Theorem \ref{thm:hyperbolic} yields the
following corollary, confirming this conjecture in all dimensions $n\ge3$ without a spin assumption. Here we identify $\partial B_{\HH^n}(R)$ with $\SSS^{n-1}(\sinh R)$ via the standard isometry.

\begin{corollary}\label{cor:hyper-map}
Under the hypotheses of Theorem~\ref{thm:hyperbolic}, let $F:(S,g_S)\longrightarrow\SSS^{n-1}(\sinh R)$
be a smooth map satisfying
$$
 \deg F\ne0,
 \qquad
 \Lip F\le1.
$$
Then $F$ is an isometry. Moreover, there exists an isometry $\overline F:(X,g)\longrightarrow
 \overline{B_{\HH^n}(R)}$
such that $\overline F|_S=\iota\circ F$, where
$\iota:\SSS^{n-1}(\sinh R)
 \longrightarrow\partial B_{\HH^n}(R)$
is the standard radial isometry.
\end{corollary}

Gromov \cite[Section 1]{GromovBoundary} established the Euclidean hyperspherical radius comparison for spin fillings, conjectured rigidity in the equality case, and asked whether the spin assumption could be removed. Theorem \ref{thm:hyperbolic} and the methods used in its proof, together with \cite[Theorem 1.4]{WWXZ}, yield the following
comparison and rigidity result without a spin assumption.
\begin{corollary}\label{cor:euclidean-radius}
Let $(X^n,g)$ be a smooth, oriented, connected, compact Riemannian
manifold of dimension $n\ge3$, with nonempty smooth connected
boundary $S=\partial X$. Let $g_S$ be the induced boundary metric,
and let $H_S$ be the mean curvature with respect to the outward
unit normal. Suppose that, for some $r>0$,
\begin{equation}\label{eq:cor-1}
 \Sc_g\ge0\quad\text{in }X,
 \qquad
 H_S\ge\frac{n-1}{r}\quad\text{on }S.
\end{equation}
Then 
\begin{equation}\label{eq:cor-2}
\rad(S,g_S)\le r,
\end{equation}
Equality holds if and only if $(X,g)$
is isometric to the closed Euclidean ball of radius $r$.
\end{corollary}

Total mean curvature comparison provides one approach to relating scalar curvature and boundary geometry, as illustrated by the Brown--York quasi-local mass \cite{BrownYork}. Shi and Tam \cite{ShiTam} established positivity and rigidity for compact oriented three-dimensional fillings with nonnegative scalar curvature and boundary components of positive Gauss and mean curvature.
For a connected boundary, its total mean curvature is bounded above by that of its Euclidean isometric embedding, with equality precisely when the filling is isometric to a
Euclidean domain. They also obtained a higher-dimensional result for spin fillings with strictly convex Euclidean boundary embeddings. Wang and Yau \cite{WangYau} and
Shi and Tam \cite{ShiTamHyperbolic} developed hyperbolic counterparts involving weighted mean curvature integrals.

Another approach studies scalar curvature through maps of nonzero degree to model spaces.
Using Dirac operator methods, Gromov and Lawson \cite{GromovLawson} showed how contracting maps give obstructions to positive scalar curvature.
Llarull \cite{Llarull} established scalar curvature rigidity of the round sphere under area-nonincreasing maps of nonzero degree from closed spin manifolds. Goette and Semmelmann \cite{GoetteSemmelmann} extended this comparison to suitable reference manifolds with nonnegative curvature operator.
For manifolds with boundary, Lott \cite{Lott} developed an index-theoretic comparison involving both scalar curvature and boundary mean curvature, under appropriate spin and
topological assumptions. More recently, Wang, Wang, Xie, and Zhu \cite[Theorem 1.1]{WWXZ} removed the spin assumption from Llarull's rigidity theorem in all dimensions.

For compact spin fillings with nonnegative scalar curvature, Gromov \cite[Section 1]{GromovBoundary} established a sharp upper bound for the minimum of the boundary mean curvature in terms of the hyperspherical radius of the boundary. Cecchini, Hirsch, and Zeidler \cite[Theorem 1.5]{CHZ} proved the rigidity statement for spin fillings.
Both comparison and rigidity also follow by combining
the lower bound of Hijazi, Montiel, and Zhang
\cite[Theorem 6]{HMZ} for the first positive eigenvalue of the boundary Dirac operator with B\"ar's sharp upper bound in terms of hyperspherical radius \cite[Main Theorem and Theorem 3]{Baer}.

The comparison and rigidity of \eqref{eq:gromov-hyperbolic} are known for spin fillings. In the hyperbolic setting, Hijazi, Montiel, and Rold\'an \cite[Theorem 2]{HMR} proved the corresponding boundary Dirac estimate under
a negative scalar curvature lower bound. Brendle, Tsiamis, and Wang \cite[Corollary 2]{BTW} gave another proof and observed that combining this estimate with B\"ar's upper bound yields \eqref{eq:gromov-hyperbolic}. Ammann and Lockman established rigidity in the equality case \cite[Theorem A]{AmmannLockman}. For more general reference domains, Sun  \cite[Theorem 2.1]{Sun} proved a spin comparison theorem for smooth parabolically convex domains in hyperbolic space.

Capillary minimization provides a geometric approach to
scalar--mean curvature comparison without a spin assumption.
Li \cite{LiPolyhedron} (with correction \cite{LiPolyhedron-2}) used capillary minimal surfaces to prove dihedral rigidity for certain classes of three-dimensional polyhedra. For smooth boundaries, Gromov \cite[Section5.8.1]{GromovLectures} proposed a dimension reduction
approach based on capillary hypersurfaces.
Following this strategy, Wang, Wang, and Zhu \cite[Theorem 1.2 and Corollary 1.3]{WWZ} established Euclidean scalar--mean rigidity and the sharp hyperspherical radius comparison, including rigidity in the equality case, in dimensions three and four.
Ko and Yao \cite[Theorem~1.4]{KoYao} developed capillary minimal slicing to prove scalar--mean rigidity for smooth convex Euclidean domains in dimension four under a rescaled mean curvature comparison
on the boundary.

Chai and Wan \cite[Theorems 1.4 and 1.10]{ChaiWanBands} used warped $\mu$-bubbles to prove rigidity theorems for spherical warped products in dimensions three and four.
Their results include rigidity of hyperbolic geodesic balls when the metric, scalar curvature, and boundary mean curvature
are bounded below by the corresponding hyperbolic quantities. Here the metric comparison is imposed throughout the ball.
Theorem \ref{thm:hyperbolic} establishes the sharp hyperspherical radius comparison and characterizes equality for arbitrary compact oriented fillings in all dimensions $n\ge3$, without imposing an interior metric comparison.

In higher dimensions, singularities of minimizing hypersurfaces obstruct direct iteration of dimension descent.
Bi and Zhu \cite{BiZhu} developed a singular dimension descent method for positive scalar curvature obstructions in arbitrary dimensions. Their approach combines weighted scalar curvature, weighted $\mu$-bubble stability, and cylindrical conformal blow up, retaining the curvature and topological information needed for successive descent steps. Related weighted conformal ideas were used by the second author and Zhu \cite{WangZhuDecay} to construct complete metrics of uniformly positive scalar curvature under a sharp supercritical quadratic decay condition.
The weighted curvature formulation and stability identities of Bi and Zhu \cite[Section 2]{BiZhu} are also particularly relevant to the argument below.

For manifolds with boundary, Wang, Wang, Xie, and Zhu
\cite{WWXZ} developed weighted scalar--mean comparison and conformal blow up methods that also control the boundary geometry of capillary hypersurfaces. These methods allow
dimension descent to continue despite interior and boundary singularities, yielding Llarull's scalar curvature rigidity theorem and scalar--mean rigidity for strictly
convex Euclidean domains in all dimensions without the spin assumption.

\subsection*{Outline of the proof.}
Our proof uses the Jang equation together with the weighted scalar curvature framework of Bi and Zhu \cite{BiZhu} and the capillary comparison theory of Wang, Wang, Xie, and Zhu \cite{WWXZ}.
We construct a proper, complete Jang graph whose boundary is isometric to $(S,g_S)$.
Integrating the Schoen--Yau identity yields a quadratic form inequality in which the boundary contribution from the divergence term combines with the mean curvature of the
graph boundary to give the sharp lower bound
$(n-1)/\sinh R$.
To treat the possibly noncompact graph, we establish the weighted obstruction in Theorem \ref{thm:obstruction}. A suitable forcing term keeps the minimizing interface in a compact region, allowing us to apply the dimension descent methods of \cite{WWXZ}.
The corresponding quadratic form obstruction,
Corollary \ref{cor:positive-form}, remains valid even when the positive interior weight tends to zero at infinity.

We prove comparison and rigidity simultaneously by assuming
$\rad(S,g_S)\ge\sinh R$.
If any of the nonnegative interior or boundary remainder terms in the graph inequality were nonzero, it would yield a strict margin and contradict the quadratic form obstruction.
This argument uses maps of nonzero degree to spheres with radii slightly smaller than $\sinh R$, so it does not require the hyperspherical radius to be attained.
All remainder terms must therefore vanish.
The resulting Hessian equation, together with the properness of the graph, forces the graph to be compact and its underlying domain
to be all of $X$.
The same equation determines the hyperbolic radial form of $g$, and Lemma \ref{lem:radial-closure} completes the identification of $(X,g)$ with the closed hyperbolic ball of radius $R$.

\subsection*{Organization of the paper.}
Section \ref{sec:preliminary} collects some preliminary results.
Section \ref{sec:obstruction} establishes the localized weighted obstruction and its quadratic form consequence.
Section \ref{sec:jang} constructs the proper Jang graph and proves the boundary estimate and the auxiliary rigidity lemmas.
Finally, Section \ref{sec:mainproofs} proves
Theorem \ref{thm:hyperbolic}, Corollary \ref{cor:hyper-map} and Corollary \ref{cor:euclidean-radius}.

\subsection*{Acknowledgements.} 
The authors would like to thank Mingxiang Li and Jintian Zhu for helpful discussions.

\subsection*{Disclosure on AI assistance.}
The authors used AI-assisted tools, principally ChatGPT. Initially, ChatGPT suggested a possible but convoluted approach limited to dimensions three through seven. Through multiple rounds of discussion with ChatGPT, incorporating ideas and arguments contributed by the authors, they finally extended the main theorem to all dimensions $n\ge3$. ChatGPT assisted with exploring intermediate arguments, performing model calculations, and examining proposed proofs for gaps. The authors directed these discussions, critically evaluated the suggestions, and independently verified the arguments and wrote the final manuscript. The authors take full responsibility for its content.

\section{Preliminaries}\label{sec:preliminary}
We record the spherical map estimates and the degree lemma used below. These facts are not claimed to be new, and some of them have related versions in \cite[Lemmas 3.1, 4.1 and 4.5]{WWXZ}.

For a possibly noncompact oriented $m$-manifold $Z$ without boundary, a smooth map $F_0:Z\to\SSS^m$ that equals $p_*$ outside a compact set has a \emph{proper degree over the punctured sphere}: it is the degree of the proper map
$Z\setminus F_0^{-1}(p_*)\rightarrow\SSS^m\setminus\{p_*\}.$ When $Z$ is compact, this agrees with its ordinary degree. For disconnected oriented boundaries, the degrees are added with the induced orientations.

Let $(S, g_S)$ be a closed oriented Riemannian manifold of dimension $m\ge 1$. The basic properties of hyperspherical radius used below are recalled from \cite[Section 2]{Baer}; see also \cite[Section 1]{GromovBoundary}. In particular, the radius defined in \eqref{eq:radius} is finite and positive, and for any constant $c>0$ we have 
$$
 \rad(S,c^2g_S)=c\rad(S,g_S).
$$
Moreover, for every smooth map $F:(S, g_S)\to\SSS^m(r)$ with $\Lip F\le 1$, 
\begin{equation}\label{eq:degree-volume}
 |\deg F|\Vol(\SSS^m(r))\le\Vol_{g_S}(S).
\end{equation}
For a linear map $A$ between finite-dimensional Euclidean spaces, the singular values of $A$ are the nonnegative square roots of the eigenvalues of the positive semi-definite operator $A^*A$.
We denote them by
$$
 \sigma_1(A)\ge \sigma_2(A)\ge\cdots\ge0.
$$
The operator norm of $A$ is
$$
 \|A\|_{\op}:=\sup_{|v|=1}|Av|=\sigma_1(A),
$$
and the trace norm of $A$ is
$$
 \|A\|_{\tr}:=\sum_i \sigma_i(A).
$$
We note that the trace norm is dual to the operator norm
$$\|A\|_{\tr}=\sup_{\|B\|_{\op}\le1}\bigl|\tr(B^*A)\bigr|.$$
Let $(Z_1, g_1)$ and $(Z_2, g_2)$ be smooth Riemannian manifolds. For a smooth map $F:Z_1\to Z_2$, set
\begin{equation}\label{eq:trace-norm}
 D_{F, g_1, g_2}=\|dF\|_{\tr, g_1, g_2}.
\end{equation}
If no confusion arises, we drop the subscripts $g_1$ and $g_2$, and simply write $D_F=D_{F, g_1, g_2}$.

Let $(S,g_S)$ be a closed oriented $m$-dimensional Riemannian manifold and $\Phi:S\to\SSS^m$ a smooth map. Fix antipodal poles $p_-,p_+$ and write, away from the poles,
$$
 \SSS^m\setminus\{p_-,p_+\}
 \cong (0,\pi)\times\SSS^{m-1},
 \qquad
 \Phi=(\theta,G),
$$
where $\theta$ is the distance from $p_-$ and $G$ is the longitude map. We have the following spherical map estimates.

\begin{lemma}\label{lem:spherical-tools}
If $Z\subset S$ is a smooth hypersurface with unit normal $\bar\eta$,
then
\begin{equation}\label{eq:polar-trace}
 D_\Phi\ge |\bar\eta\theta|+(\sin\theta)D_G,
 \qquad \text{ for } 0<\theta<\pi .
\end{equation}
Here $D_G$ is computed on $Z$ with its induced metric and round target metric. Moreover, for every $\eps>0$ there is a smooth degree-one map
$$
 P(\theta,z)=(a(\theta),z):\SSS^m\to\SSS^m
$$
which is constant near the two poles and satisfies $\Lip P\le1+\eps$.
Hence
\begin{equation}\label{eq:polar-preparation}
 D_{P\circ\Phi}\le(1+\eps)D_\Phi,
 \qquad
 \deg(P\circ\Phi)=\deg\Phi .
\end{equation}
\end{lemma}

The estimate \eqref{eq:polar-trace} is the trace norm calculation in \cite[proof of Lemma 4.5]{WWXZ}.  The estimate \eqref{eq:polar-preparation} is the quantitative version of the pole preparation used in \cite[Lemma 3.1]{WWXZ}. For the reader's convenience, we sketch the proof here.

\begin{proof}
Fix a point $x\in Z$ with $0<\theta<\pi$.  Using the polar decomposition of the round metric,
$$
 g_{\SSS^m}=d\theta^2+\sin^2\theta\,g_{\SSS^{m-1}},
$$
split
$$
 T_xS=\mathbb R\bar\eta\oplus T_xZ,
 \qquad
 T_{\Phi(x)}\SSS^m
 =\mathbb R\partial_\theta\oplus\mathcal H,
$$
where $\mathcal H={\mathrm{Ker}}\, d\theta$ is the angular subspace. We use the same symbol $\theta$ for the polar angle on $\SSS^m$ and its pullback by $\Phi$. Set $e_0=\bar\eta$. Choose orthonormal bases
$e_1,\ldots,e_{m-1}$ of $T_xZ$ and
$E_1,\ldots,E_{m-1}$ of $\mathcal H$ so that $$\langle d\Phi|_{T_xZ}(e_i), E_i\rangle=\sigma_i,\qquad i=1,\cdots,m-1,$$ where $\sigma_1, \cdots, \sigma_{m-1}$ are singular values of $\pi_{\mathcal H}\circ d\Phi|_{T_xZ}$.  Choose a unit vector
$E_0\in\mathbb R\partial_\theta$
so that
$$
 \langle d\Phi(e_0),E_0\rangle
 =|\bar\eta\theta|.
$$
Define the linear map $Q:T_xS\to T_{\Phi(x)}\SSS^m$ by $Qe_i=E_i$, $0\le i\le m-1$. Then $\|Q\|_{\op}=1$.  Hence,
\begin{align}\label{eq:lem2-1}
 D_\Phi=\|d\Phi\|_{\tr}\ge \tr(Q^*d\Phi)
 =\sum_{i=0}^{m-1}\langle d\Phi(e_i),E_i\rangle
 =|\bar\eta\theta|
   +\|\pi_{\mathcal H}\circ d\Phi|_{T_xZ}\|_{\tn}.
\end{align}
Since $\Phi=(\theta,G)$, the polar expression for
the round metric shows that the singular values of
$\pi_{\mathcal H}\circ d\Phi_x|_{T_xZ}$ are
$\sin\theta(x)$ times those of $d(G|_Z)_x$, where
the latter are computed using the round metric
on $\SSS^{m-1}$. Consequently,
$$
 \|\pi_{\mathcal H}\circ d\Phi|_{T_xZ}\|_{\tr}
 =(\sin\theta)D_G.
$$
By \eqref{eq:lem2-1}, this proves \eqref{eq:polar-trace}.

For the second assertion, choose $\delta>0$ small and $k\in C^\infty([0,\pi])$ such that
$$
\left\{
\begin{aligned}
&0\le k\le1,\\
&k(\pi-s)=k(s),\\
&k(s)=0
&&\text{for }s\in[0,\delta]\cup[\pi-\delta,\pi],\\
&k'(s)\ge0
&&\text{for }s\in[0,\frac\pi2],\\
&k(s)=1
&&\text{for }s\in[2\delta,\pi-2\delta].
\end{aligned}
\right.
$$
Set
$$
 c=\frac{\pi}{\int_0^\pi k(s)\,ds},
 \qquad
 a(\theta)=c\int_0^\theta k(s)\,ds .
$$
By choosing $\delta$ sufficiently small, we may
assume
$1\le c\le1+\eps.$
Moreover,
$$
 a(0)=0,\qquad a(\pi)=\pi,\qquad a'(\theta)=ck(\theta)\le c,
$$
and $a$ is constant near both endpoints. We note that
$$
 \sin a(\theta)\le\sin\theta
 \qquad \text{ for } 0\le\theta\le\pi.
$$
Define
$P(\theta,z)=(a(\theta),z).$
Since $a$ is constant near $0$ and $\pi$, this extends smoothly across the two poles, and is constant in neighborhoods of them. For $0<\theta<\pi$, its radial
and angular singular values are
$a'(\theta)$ and $\frac{\sin a(\theta)}{\sin\theta}$, respectively. Since $dP=0$ near the poles,
$$
 \Lip P
 \le
 \max\left\{
       \sup_{0\le \theta\le\pi} a'(\theta),
       \sup_{0<\theta<\pi}\frac{\sin a(\theta)}{\sin\theta}
      \right\}
 \le c
 \le1+\eps.
$$
The map $P$ is homotopic to the identity, and hence $\deg P=1.$
Finally, we have
\begin{align*}
 &D_{P\circ\Phi}
 =\|dP\circ d\Phi\|_{\tn}
 \le
 \|dP\|_{\op}\|d\Phi\|_{\tn}
 \le(1+\eps)D_\Phi,\\
 &\deg(P\circ\Phi)
 =
 \deg P\,\deg\Phi
 =
 \deg\Phi.
\end{align*}
This proves \eqref{eq:polar-preparation}.
\end{proof}

The following degree descent statement is the topological content of \cite[Lemma 4.5]{WWXZ}. We record it in the finite perimeter form needed here. This lemma relates the degree of $\Phi$ to the signed degree of the longitude map $G$ on the regular contact boundary.

\begin{lemma}\label{lem:longitude-degree}
Let $(S, g_S)$ be a closed oriented $m$-dimensional Riemannian manifold, $m\ge2$, and $\Phi:S\to\SSS^m$ a smooth map. Fix antipodal poles $p_-,p_+$ and write $\Phi=(\theta,G)$ away from them. Let $U\subset S$ be an open set of finite perimeter such that
$$
 \Phi^{-1}(p_-)\Subset U,
 \qquad
 \overline U\cap\Phi^{-1}(p_+)=\varnothing.
$$
Assume that $\partial U\setminus K$ is a smooth embedded hypersurface
and that $U$ lies locally on one side of it, where
$K\subset\partial U$ is compact and its Hausdorff dimension satisfies
$\dim K<m-1$. Set $Z=\partial U\setminus K$ with its induced orientation. Then
$$
 |\deg(G: Z\to\SSS^{m-1})|=|\deg\Phi|,
$$
where the degree of $G$ is computed at a regular value outside $G(K)$. If $K=\varnothing$, no dimension assumption on $K$ is needed.
\end{lemma}

\begin{proof}
Since $\partial U$ is compact and avoids
$\Phi^{-1}(p_-)\cup\Phi^{-1}(p_+)$, the longitude map $G$ is smooth
on a neighborhood of $\partial U$. Since $G$ is smooth near the compact set $K$ and $\dim K<m-1$, we have $\mathcal H^{m-1}(G(K))=0.$
Choose a regular value $q\in\SSS^{m-1}\setminus G(K)$ of $G|_Z$, and a neighborhood $V_0$ of $q$ such that $\overline{V_0}\cap G(K)=\varnothing.$
Since $\partial U$ is compact, 
$G^{-1}(\overline{V_0})\cap\partial U$
is compact and disjoint from $K$, hence is a compact subset of $Z$. The critical set of $G|_Z$ in this compact set has compact image, which does not contain $q$. We may choose a connected coordinate neighborhood $V$ of $q$ with $V\Subset V_0$ consisting entirely of regular values of $G|_Z$, with
${G|_Z}^{-1}(\overline V)$ compact in $Z$. Let $\omega$ be a smooth $(m-1)$-form supported in $V$ and normalized by
$\int_{\SSS^{m-1}}\omega=1.$
Then
\begin{equation}\label{eq:lem2-2-1}
 \int_Z G^*\omega=\deg_q G,
\end{equation}
where $\deg_qG$ denotes the signed degree computed at $q$.

We next construct a top degree form on $\SSS^m$ which computes $\deg\Phi$ and whose pullback is supported in $U$. Since $\Phi^{-1}(p_-)\Subset U$, by compactness there is a constant $0<b<\pi$ such that $\Phi^{-1}(\{\theta\le b\})\subset U$. Let
$$
 P_{\rm lon}:\SSS^m\setminus\{p_-,p_+\}\rightarrow\SSS^{m-1}
$$
be the longitude projection, so that $\Phi=(\theta,G)$ and
$P_{\rm lon}\circ\Phi=G$ away from the poles.  Choose a smooth function $\lambda: [0,\pi]\rightarrow [0,1]$ such that
$$
 \lambda=0 \quad\text{near }0,
 \qquad
 \lambda=1 \quad\text{on }[b,\pi],
$$
and set
$$
 \alpha=\lambda(\theta)\,P_{\rm lon}^*\omega
 \qquad\text{on }\SSS^m\setminus\{p_-,p_+\}.
$$
The form $\alpha$ extends smoothly across $p_-$, and
$$
 d\alpha
   =\lambda'(\theta)\,d\theta\wedge P_{\rm lon}^*\omega
$$
extends smoothly over all of $\SSS^m$, since it vanishes near both poles. Orient $\SSS^{m-1}$ so that
$d\theta\wedge \operatorname{vol}_{\SSS^{m-1}}$
agrees with the given orientation of $\SSS^m$. Then,
$$
 \int_{\SSS^m}d\alpha
 =
 \left(\int_0^\pi\lambda'(\theta)\,d\theta\right)
 \left(\int_{\SSS^{m-1}}\omega\right)
 =1.
$$
Hence the degree formula gives
$$
 \deg\Phi=\int_S\Phi^*(d\alpha).
$$
Moreover,
$$
 \supp\Phi^*(d\alpha)
 \subset\Phi^{-1}(\{\theta\le b\})\subset U,
$$
and therefore
\begin{equation}\label{eq:lem2-2-2}
 \deg\Phi=\int_U\Phi^*(d\alpha).
\end{equation}
It remains to move this integral to the boundary of $U$.  Since
$\overline U\cap\Phi^{-1}(p_+)=\varnothing$, we can choose
$\chi\in C^\infty(S)$ such that $\chi=1$ on a neighborhood of
$\overline U$ and $\chi=0$ near $\Phi^{-1}(p_+)$.  Define $\zeta=\chi\,\Phi^*\alpha$. It is a globally smooth $(m-1)$-form on $S$, and on a neighborhood of
$\overline U$, $d\zeta=\Phi^*(d\alpha).$
Therefore,
\begin{equation}\label{eq:lem2-2-3}
 \deg\Phi
 =\int_U d\zeta
 =\int_{\partial^*U}\zeta.
\end{equation}
Since $U$ lies locally on one side of
$Z=\partial U\setminus K$, we have
$
 \partial^*U\setminus K=Z
$
up to an $\mathcal H^{m-1}$-null set. By $\mathcal H^{m-1}(K)=0$, the exceptional set $K$ contributes
nothing to the boundary integral. On $Z$ we have $\theta>b$, hence $\lambda=1$ and $\chi=1$, so
$$
 \zeta|_Z
   =\Phi^*(P_{\rm lon}^*\omega)|_Z
   =G^*\omega.
$$
Consequently, by \eqref{eq:lem2-2-1} and \eqref{eq:lem2-2-3}
$$
 \deg\Phi
 =\int_ZG^*\omega
 =\deg_qG,
$$
up to the fixed sign determined by the orientation conventions.
Thus,
$$
 \bigl|\deg G\bigr|=|\deg\Phi|.
$$

If $K=\varnothing$, the same argument applies without any dimension
assumption on $K$.
\end{proof}

We finish this section with a lemma on proper functions.
\begin{lemma}\label{lem:exhaustion}
Let $(M,g)$ be a complete Riemannian manifold with smooth compact boundary. For every compact
set $K\subset M$ containing a boundary collar, there is a smooth proper function $\sigma\ge0$ with $|d\sigma|\le 1$ and $\sigma=0$ on a neighborhood of $K$.
\end{lemma}
\begin{proof}
This follows from \cite[Lemma 4.1]{WWXZ}.  That lemma provides a smooth, proper, nonnegative function $\varphi$ on $M$ such that
$$
 |d\varphi|\le1,
 \qquad
 \partial_\nu \varphi=0\quad\text{on }\partial M,
$$
where $\nu$ is the outward unit normal.
Since $K$ is compact, choose $A>\sup_K \varphi$ and a smooth nondecreasing function $\chi: [0,\infty)\to[0,\infty)$ such that
$$
 \chi=0\quad\text{on }[0,A],\qquad
 0\le\chi'\le1,
$$
and $\chi(t)\to\infty$ as $t\to\infty$.
Set $\sigma=\chi\circ \varphi.$ It is smooth, nonnegative, and proper. Moreover, it satisfies
$$
 |d\sigma|=\chi'(\varphi)|d\varphi|\le1,
$$
and vanishes on the open set $\{\varphi<A\}$ containing $K$.
\end{proof}

\section{Localizing the weighted obstruction}\label{sec:obstruction}

Throughout this section, $(M^n,g)$ may be a noncompact Riemannian manifold,
while its smooth boundary $S=\partial M$ is compact.
Here $g$ denotes the metric on $M$, and no relation is assumed
between this metric and the one denoted by the same symbol
in $(X,g)$.
Unsubscripted covariant derivatives, gradients, and
divergences are taken with respect to $g$.
Completeness is understood with respect to the intrinsic Riemannian
length metric on $M$, with its boundary included. In particular,
functions in $C_c^\infty(M)$ may have nonzero boundary trace.

Let $\nu$ be the outward unit normal of $S$. For $f\in C^\infty(M)$,
the weighted mean curvature of $S$ with respect to $e^{-f}$ is
\begin{equation}\label{eq:weighted-mean}
 H_f=H_S-\nu f.
\end{equation}

Set
\begin{equation}\label{eq:weighted-notation}
 a_n:=\frac{n}{n+1},
 \qquad
 S_n(g,f)
 :=\Sc_g+2\Delta_g f-a_n|df|_g^2,
\end{equation}
where $\Delta_g=\diver_g\nabla_g$. 
The purpose of this section is to convert the pointwise inequalities
$$
 S_n(g,f)>0 \quad\text{on }M,
 \qquad
 H_f>D_\Phi \quad\text{on }S,
$$
whose interior positivity need not be uniform, into a compact capillary interface on which the associated quadratic form has uniform positive interior and boundary lower bounds. At the same time, the nonzero degree of $\Phi$ is transferred to the longitude map on the contact boundary. These
properties allow the dimension descent argument of \cite{WWXZ} to be applied.

\begin{theorem}\label{thm:obstruction}
Let $n\ge3$, and let $(M^n,g)$ be a smooth, oriented, connected and complete Riemannian manifold  with nonempty smooth compact boundary $S$. There are no smooth functions $f\in C^\infty(M)$ and smooth maps
$\Phi:S\to\SSS^{n-1}$ such that
\begin{equation}\label{eq:obstruction-data}
 S_n(g,f)>0\quad\text{on }M,
 \qquad
 H_f>D_\Phi\quad\text{on }S, \qquad \deg\Phi\ne0.
\end{equation}
The first inequality is pointwise, and no uniform positive lower bound is assumed.
\end{theorem}

The proof of Theorem \ref{thm:obstruction} is deferred to the end of this section. 

For a smooth function $h$ on $M$ and a smooth function $\beta$ on $S$, with $|\beta|\le1$, define the functional
\begin{equation}\label{eq:energy}
 \cF_{f,h,\beta}(E)
 =\int_{\partial^*E\cap M^\circ}e^{-f}
  -\int_S\beta\,\chi_Ee^{-f}
  -\int_Mh\,\chi_Ee^{-f},
\end{equation}
where $E\subset M$ is a set of finite perimeter
for which the displayed integrals are finite. We denote by $\partial^*E$ the reduced boundary of $E$, and by $\chi_E$ its characteristic function. All integrals are taken with respect to the induced Riemannian volume or area measure.  In the boundary integral over $S=\partial M$, $\chi_E$ is understood in the sense of its $BV$ trace on $S$.
On a noncompact manifold, local minimization is understood in terms of relative energy for compactly supported variations. The constrained minimization below takes place in a compact container, where all terms are finite.
The function denoted by $f$ in \cite{WWXZ} corresponds to $-f$ in our notation. Set $\rho=e^{-f}$, $\Psi=h$, and $\mu_\partial=\beta$.
Then \eqref{eq:energy} is the functional in \cite[(4.3)]{WWXZ}.

Let $E$ be a locally minimizing set for \eqref{eq:energy}, and let $Y=\overline{\partial^*E\cap M^\circ}$ be the closure of its free interface. Denote its regular part by
$Y_{\rm reg}$ and its induced metric by $g_Y$. For the variation formulas below, set $\beta=\cos\theta$,
where $\theta$ is smooth on a neighborhood of
$\partial Y_{\rm reg}$ in $S$ and satisfies
$$0<\theta<\pi \qquad\text{ on } \partial Y_{\rm reg}.$$
Assume that $Y_{\rm reg}$ is two-sided, and let $N$ be its unit normal pointing out of $E$. Along $\partial Y_{\rm reg}$, let $\eta$ be the outward unit conormal in $Y_{\rm reg}$, and let $\bar\eta$ be the outward unit normal in $S$ to the region where the $BV$ trace of $\chi_E$ equals one. Recall that $\nu$ denotes the outward unit normal of $S=\partial M$.

Let $A_Y$ and $H_Y$ denote the second fundamental form and mean
curvature of $Y_{\rm reg}$ with respect to $N$, namely,
$$
 A_Y(V,Z)=g(\nabla_VN,Z),
 \qquad
 H_Y=\tr_{g_Y}A_Y.
$$
We use the first and second variation formulas from
\cite[Lemma 4.4]{WWXZ}. The first variation of \eqref{eq:energy}
gives the prescribed weighted mean curvature equation and the
capillary boundary condition:
\begin{align}
 H_Y-Nf&=h
 \qquad\text{on }Y_{\rm reg},\label{eq:first-1}\\
 \cos\theta&=-\langle N,\nu\rangle
 \qquad\text{on }\partial Y_{\rm reg}.\label{eq:first-2}
\end{align}
Thus the prescribed angle $\theta$ is the angle between $N$ and
$-\nu$, or, equivalently, between $\eta$ and $\bar\eta$.
With these choices of normals, we have
\begin{equation}\label{eq:capillary-normals}
 N=(\sin\theta)\bar\eta-(\cos\theta)\nu,
 \qquad
 \eta=(\cos\theta)\bar\eta+(\sin\theta)\nu.
\end{equation}

Set $f_Y=f|_{Y_{\rm reg}}$. Along $\partial Y_{\rm reg}$, let
$H_{\partial Y}$ be the mean curvature of
$\partial Y_{\rm reg}\subset Y_{\rm reg}$ with respect to $\eta$,
and define
\begin{equation}\label{eq:slice-mean}
 H_{\partial Y}^{f_Y}
 :=H_{\partial Y}-\eta(f_Y).
\end{equation}
The second variation gives the stability inequality
\begin{equation}\label{eq:stability}
 \int_{Y_{\rm reg}}e^{-f_Y}
 \bigl(|dv|_{g_Y}^2-P_fv^2\bigr)
 +\int_{\partial Y_{\rm reg}}e^{-f_Y}B_fv^2
 \ge0
\end{equation}
for every $v\in C_c^\infty(Y_{\rm reg})$, where
\begin{align}
 P_f
 &=\Ric_g(N,N)+\Hess_g f(N,N)+|A_Y|_{g_Y}^2+Nh,
 \label{eq:Pf}\\
 B_f
 &=H_{\partial Y}^{f_Y}
   -\frac{H_f+h\cos\theta-\bar\eta\theta}
          {\sin\theta}.
 \label{eq:Bf}
\end{align}
Here $Y_{\rm reg}$ is regarded as a manifold with boundary
$\partial Y_{\rm reg}$. Thus a function
$v\in C_c^\infty(Y_{\rm reg})$ may have nonzero boundary trace, while its compact support stays away from the singular set.

Assume that the boundary map $\Phi$ has the polar representation $\Phi=(\theta,G)$ on a neighborhood of $\overline{\partial Y_{\rm reg}}$ in $S$, where $\theta$ is the same prescribed angle as in $\beta=\cos\theta$ and $G$
is the longitude map. For $v\in C_c^\infty(Y_{\rm reg})$, define
\begin{equation}\label{eq:slice-form}
\begin{split}
 \Q_Y(v)={}&
 \int_{Y_{\rm reg}}e^{-f_Y}
 \left[
 |dv|_{g_Y}^2
 +\frac12
 \left(
 \Sc_{g_Y}+2\Delta_{g_Y}f_Y
 -a_n|df_Y|_{g_Y}^2
 \right)v^2
 \right]\\
 &+\int_{\partial Y_{\rm reg}}e^{-f_Y}
 \bigl(H_{\partial Y}^{f_Y}-D_G\bigr)v^2.
\end{split}
\end{equation}
Here $D_G$ is the trace norm of $d(G|_{\partial Y_{\rm reg}})$, using the induced metric on $\partial Y_{\rm reg}$ and the round metric on $\SSS^{n-2}$.

The next inequality is a quantitative form of the capillary descent calculation in \cite[Section 4.2]{WWXZ}, specialized to the coefficient $a_n=n/(n+1)$ and arranged for the localization argument below.

\begin{lemma}\label{lem:variation}
Let $E$ be a locally minimizing set for \eqref{eq:energy}, and $Y=\overline{\partial^*E\cap M^\circ}$ be the capillary interface with regular part $Y_{\rm reg}$. Suppose that $Y$ and $Y_{\rm reg}$ satisfy the assumptions above. Then, for every $v\in C_c^\infty(Y_{\rm reg})$,
\begin{equation}\label{eq:descent-margin}
\begin{split}
 \Q_Y(v)\ge{}&
 \frac12\int_{Y_{\rm reg}}e^{-f_Y}
 \left(S_n(g,f)+2Nh+\frac12h^2\right)v^2\\
 &+\int_{\partial Y_{\rm reg}}e^{-f_Y}
 \frac{H_f-D_\Phi+h\cos\theta}{\sin\theta}\,v^2 .
\end{split}
\end{equation}
\end{lemma}

\begin{proof}
Set $u=Nf.$ Adding and subtracting the stability coefficients
$P_f$ and $B_f$ in the definition of $\Q_Y$, and using
\eqref{eq:stability}, we obtain
$$
\begin{aligned}
 \Q_Y(v)\ge{}&
 \int_{Y_{\rm reg}}e^{-f_Y}
 \left[
 P_f+\frac12
 \left(
 \Sc_{g_Y}+2\Delta_{g_Y}f_Y-a_n|df_Y|_{g_Y}^2
 \right)
 \right]v^2\\
 &+\int_{\partial Y_{\rm reg}}e^{-f_Y}
 \left(H_{\partial Y}^{f_Y}-D_G-B_f\right)v^2.
\end{aligned}
$$
We estimate the two coefficients on the right separately.

For the interior term, the Gauss equation 
$$
 \Sc_{g_Y}
 =\Sc_g-2\Ric_g(N,N)+H_Y^2-|A_Y|_{g_Y}^2
$$
and
$$
 \Delta_{g_Y}f_Y
 =
 \Delta_gf-\Hess_gf(N,N)-H_YNf
$$
give
\begin{align}
 2P_f
 +\Sc_{g_Y}+2\Delta_{g_Y}f_Y
      -a_n|df_Y|_{g_Y}^2 
 =
 S_n(g,f)
 +|A_Y|^2+H_Y^2-2H_Yu+a_nu^2+2Nh.
 \label{eq:gauss-weighted}
\end{align}
By \eqref{eq:first-1},
\begin{equation}\label{eq:lem-3-2-1}
 H_Y=h+u.
\end{equation}
Note $\dim Y=n-1$. We write
$$
 A_Y^\circ
 :=
 A_Y-\frac{H_Y}{n-1}g_Y
$$
and set
$$
 b_n
 :=
 a_n-1+\frac1{n-1}
 =
 \frac{2}{n^2-1}>0.
$$
Using
$$
 |A_Y|^2
 =
 |A_Y^\circ|^2+\frac{H_Y^2}{n-1},
$$
and \eqref{eq:lem-3-2-1}, we have
\begin{equation}\label{eq:trace-square}
 |A_Y|^2+H_Y^2-2H_Yu+a_nu^2
 =
 |A_Y^\circ|^2
 +b_n\left(
      u+\frac{h}{(n-1)b_n}
     \right)^2
 +\frac12h^2\ge\frac12 h^2.
\end{equation}
Hence, by \eqref{eq:gauss-weighted},
\begin{equation}\label{eq:lem-interior}
 2P_f
 +\Sc_{g_Y}+2\Delta_{g_Y}f_Y-a_n|df_Y|_{g_Y}^2
 \ge
 S_n(g,f)+2Nh+\frac12h^2.
\end{equation}

For the boundary term, by \eqref{eq:Bf},
$$
 H_{\partial Y}^{f_Y}-D_G-B_f
 =
 \frac{H_f+h\cos\theta-\bar\eta\theta}
      {\sin\theta}
 -D_G.
$$
By the polar trace estimate \eqref{eq:polar-trace},
$$
 D_\Phi
 \ge
 |\bar\eta\theta|+(\sin\theta)D_G,
$$
and therefore
\begin{equation}\label{eq:lem-boundary}
 H_{\partial Y}^{f_Y}-D_G-B_f
 \ge
 \frac{H_f-D_\Phi+h\cos\theta}{\sin\theta}.
\end{equation}
Combining \eqref{eq:lem-interior} and \eqref{eq:lem-boundary} proves \eqref{eq:descent-margin}.
\end{proof}

\begin{lemma}\label{lem:shielding}
Let $n\ge3$, and let $(M^n,g)$ be a smooth, oriented, connected and complete Riemannian manifold  with nonempty smooth compact boundary $S$. Suppose that $f\in C^\infty(M)$ and $\Phi:S\to\SSS^{n-1}$ is a smooth map satisfying \eqref{eq:obstruction-data}.
Let $p_-$ and $p_+$ be antipodal regular values of $\Phi$.
Then there exist a degree-one map $P$ as in
Lemma~\ref{lem:spherical-tools} that is constant near
$p_-$ and $p_+$, a set $E\subset M$ of finite perimeter,
and a function $h\in C^\infty(M)$ vanishing on a neighborhood of $S$, such that the following properties hold. Here $\theta$ and $G$
denote the polar angle and longitude of $P\circ\Phi$, respectively.
\begin{enumerate}[label=\textup{(\roman*)},leftmargin=*]

\item\label{itemi}
The set $E$ minimizes $\cF_{f,h,\beta}$ in the constrained class used in its construction, with $\beta=\cos\theta$. Moreover, the auxiliary constraints are inactive near $Y:=\overline{\partial^*E\cap M^\circ}$. 
\item\label{itemii}
$Y$ is compact, and $Y_{\rm reg}$ is two-sided and oriented by the unit normal pointing out of $E$. Moreover, for some $\theta_0>0$,
$$
 \theta_0\le \theta\le \pi-\theta_0
 \qquad\text{on }\overline{\partial Y_{\rm reg}}.
$$

\item\label{itemiii}
The longitude map $G:\partial Y_{\rm reg}\rightarrow\SSS^{n-2}$
extends smoothly to a neighborhood of
$\overline{\partial Y_{\rm reg}}$ in $S$, and has nonzero degree in the sense of
Lemma \ref{lem:longitude-degree}, with $|\deg G|=|\deg\Phi|\ne0$.

\item\label{itemiv}
Let $g_Y$ be the induced metric on $Y_{\rm reg}$, and $f_Y:=f|_{Y_{\rm reg}}$. There exist constants $c_0,c_1>0$ such that
\begin{equation}\label{eq:slice-margins}
 \Q_Y(v)\ge
 c_0\int_{Y_{\rm reg}}e^{-f_Y}v^2
 +
 c_1\int_{\partial Y_{\rm reg}}e^{-f_Y}v^2
\end{equation}
for every $v\in C_c^\infty(Y_{\rm reg})$.

\end{enumerate}
\end{lemma}

\begin{remark}
By item \ref{itemi}, there exists a neighborhood $U$ of $Y$ such that $E$ minimizes $\cF_{f,h,\beta}$ among all finite perimeter competitors $E'$ satisfying $E'\triangle E\Subset U.$
\end{remark}

\begin{proof}
{\bf Proof of item \ref{itemi}.}
We first choose the map $P$ and construct the two pole collars.
Set
$$
 \delta_S:=\min_S(H_f-D_\Phi)>0,
$$
and choose $\varepsilon>0$ such that
$\epsilon\max_S D_\Phi<\delta_S/2$.
Since $p_-$ and $p_+$ are regular values of $\Phi$ and $S$ is compact, we may choose $r_*>0$ sufficiently small that the inverse image of each closed geodesic ball $\overline {B_{2r_*}(p_\pm)}$ is a finite disjoint union of sets on which $\Phi$ is a diffeomorphism onto  $\overline {B_{2r_*}(p_\pm)}$. We also require $2r_*<\pi/2$.
Following the construction in Lemma \ref{lem:spherical-tools}, choose a smooth cutoff function $k$ to vanish exactly on $[0,r_*]\cup[\pi-r_*,\pi]$ and to be positive
on $(r_*,\pi-r_*)$. 
Define
$$
 a(s)=
 \frac{\pi\int_0^s k(t)\,dt}
      {\int_0^\pi k(t)\,dt},
 \qquad 0\le s\le\pi.
$$
In polar coordinates centered at $p_-$, define $P$ by
$$
 P(s,z)=(a(s),z),
 \qquad 0<s<\pi,\quad z\in\mathbb{S}^{n-2}.
$$
Thus $P$ changes the polar angle from $s$ to $a(s)$
and leaves the longitude map unchanged.
Since $a=0$ near $0$ and $a=\pi$ near $\pi$,
the map $P$ extends smoothly across both poles.
It is homotopic to the identity and therefore has degree one.
By taking $r_*$ sufficiently small in this construction,
we may ensure that $\operatorname{Lip} P\le 1+\varepsilon$.
By Lemma \ref{lem:spherical-tools} and the choice of $\varepsilon$,
we have
$$
 H_f-D_{P\circ\Phi}
 \ge H_f-(1+\varepsilon)D_\Phi
 \ge\frac{\delta_S}{2}>0,
 \qquad
 \deg(P\circ\Phi)=\deg\Phi.
$$
Define
$$
 \beta(x)=\langle P(\Phi(x)), \,p_-\rangle_{\mathbb R^n},
 \qquad x\in S.
$$
Then $\beta\in C^\infty(S)$ and $|\beta|\le1$.
Away from the pole preimages, write $P\circ\Phi=(\theta,G)$, where $\theta$ is the polar angle measured from $p_-$. In these coordinates, $\beta=\cos\theta$.

For $ x\in\Phi^{-1}\bigl(\overline{B_{2r_*}(p_\pm)}\bigr)$, define
$$
 s_\pm(x)=d_{\SSS^{n-1}}(p_\pm,\Phi(x)).
$$
Then the functions $s_\pm^2$ are smooth on a neighborhood
of $\Phi^{-1}\bigl(\overline {B_{2r_*}(p_\pm)}\bigr)$. By the construction of $P$, we have
$P^{-1}(p_\pm)=\overline {B_{r_*}(p_\pm)}$. Consequently,
$$
 (P\circ\Phi)^{-1}(p_\pm)=\Phi^{-1}\bigl(\overline {B_{r_*}(p_\pm)}\bigr)=\{s_\pm\le r_*\}.
$$
In particular,
\begin{equation}\label{eq:lem3-3-beta}
 \beta=1\quad\text{on }(P\circ\Phi)^{-1}(p_-),
 \qquad
 \beta=-1\quad\text{on }(P\circ\Phi)^{-1}(p_+).
\end{equation}
Note the sets $\Phi^{-1}\bigl(\overline {B_{2r_*}(p_\pm)}\bigr)$
are compact subsets of $S$.
Since $d\Phi$ is invertible on these sets, there exists
a constant $c_*>0$ such that
$$
 |ds_\pm|_{g_S}\ge c_*
 \qquad\text{where }0<s_\pm\le2r_*.
$$
Since $H_f>D_\Phi\ge0$ and $S$ is compact, we may choose $\kappa>0$ such that
\begin{equation}\label{eq:pole-H-gap}
 H_f\ge2\kappa\qquad\text{on }S.
\end{equation}
Choose an inward normal collar $S\times[0,t_*)$, where $t$
is the inward distance from $S$. In these coordinates,
$$
 g=dt^2+g_t,
$$
where $g_t$ is the induced metric on the slice $S\times\{t\}$.
In particular, $g_0=g_S$ and $\partial_t=-\nu$ along $S$.
For a vector field $Q$, write
$\diver_fQ:=\diver Q-\langle\nabla f,Q\rangle$.
Along $S$ we have
$$
 \diver_f\partial_t=-H_S+\nu f=-H_f.
$$
By continuity of $\diver_f\partial_t$ and compactness of $S$, after decreasing $t_*$
if necessary, we have
\begin{equation}\label{eq:pole-collar-sign}
 \diver_f\partial_t\le-\kappa
 \qquad\text{on }S\times[0,t_*).
\end{equation}

Extend $s_\pm^2$ in the normal direction by setting
$s_\pm^2(x,t)=s_\pm^2(x)$.
For $\tau>0$, define the vector fields
$$
 Q_\pm=
 \frac{\partial_t+\tau\nabla_{g_t}(s_\pm^2)}
 {\sqrt{1+\tau^2|d(s_\pm^2)|_{g_t}^2}}
$$
on $\Phi^{-1}(B_{2r_*}(p_\pm))\times[0,t_*)$.
These are smooth unit vector fields and converge uniformly
in $C^1$ to $\partial_t$ as $\tau\downarrow0$. We may therefore fix $\tau>0$ sufficiently small that
$\tau(2r_*)^2<t_*$ and
\begin{equation}\label{eq:pole-Q-divergence}
 \diver_fQ_\pm\le-\frac{\kappa}{2}
 \qquad\text{on }
 \Phi^{-1}(B_{2r_*}(p_\pm))\times[0,t_*).
\end{equation}
Since $a(r_*)=0$, we may now choose $b\in(r_*,2r_*)$ sufficiently
close to $r_*$, such that
$$
 a(b)<\arctan(2\tau r_*c_*).
$$
Set
$$
 O_\pm=\{s_\pm<b\},\qquad
 \psi_\pm=b^2-s_\pm^2,\qquad
 V_\pm=\{(x,t):x\in O_\pm,\ 0\le t<\tau\psi_\pm(x)\}.
$$
The sets $O_-$ and $O_+$ have disjoint closures and smooth boundaries.
Moreover, $\psi_\pm>0$ in $O_\pm$, $\psi_\pm=0$ on
$\partial O_\pm$, and $d\psi_\pm\ne0$ on $\partial O_\pm$. The interior face of
$V_\pm$ is the graph
$$
 \Gamma_\pm:=\{(x,t):x\in O_\pm,\ t=\tau\psi_\pm(x)\},
$$
and $Q_\pm$ is its outward unit normal. Notice that
$d\psi_\pm=-d(s_\pm^2)$ is independent of $b$, so the final choice
of $b$ does not affect \eqref{eq:pole-Q-divergence}.

Along $S$, we have
$$
 \langle Q_\pm,\nu\rangle
 =-\frac{1}{\sqrt{1+\tau^2|d\psi_\pm|_{g_S}^2}},
 \qquad
 \beta|_{O_-}=\cos a(s_-),\quad
 \beta|_{O_+}=-\cos a(s_+).
$$
If $s_\pm\le r_*$, then $a(s_\pm)=0$. If $r_*<s_\pm<b$, then
$$
 a(s_\pm)\le a(b)<\arctan(2\tau r_*c_*)
 \le\arctan\bigl(\tau|d\psi_\pm|_{g_S}\bigr),
$$
because $|d\psi_\pm|_{g_S}=2s_\pm|ds_\pm|_{g_S}\ge2r_*c_*$.
Thus, in both cases,
$$
 \cos a(s_\pm)\ge
 \frac{1}{\sqrt{1+\tau^2|d\psi_\pm|_{g_S}^2}}.
$$
Consequently,
\begin{equation}\label{eq:pole-wall-signs}
 \langle Q_-,\nu\rangle+\beta\ge0
 \quad\text{on }V_-\cap S,
 \qquad
 \langle Q_+,\nu\rangle-\beta\ge0
 \quad\text{on }V_+\cap S.
\end{equation}
Since $\{s_\pm\le r_*\}\Subset O_\pm$ and
$\psi_\pm\ge b^2-r_*^2>0$ on $\{s_\pm\le r_*\}$, we can choose smooth collar neighborhoods
$C_\pm\Subset V_\pm$ such that
$$
 (P\circ\Phi)^{-1}(p_\pm)
 \subset\operatorname{int}_S(C_\pm\cap S).
$$
We impose the constraints
\begin{equation}\label{eq:polar-set-constraints}
 \chi_E=1\quad\text{a.e. on }C_-,
 \qquad
 \chi_E=0\quad\text{a.e. on }C_+.
\end{equation}

{\bf We now construct a constrained minimizer. The auxiliary constraints will be shown below to be inactive near the resulting capillary interface.}

\noindent{\bf Case (1): $M$ is compact. } In this case, take $h\equiv0$ and minimize
$\cF_{f,h,\beta}$, with $\beta=\cos\theta$, among sets of finite perimeter satisfying \eqref{eq:polar-set-constraints}. Such competitors
exist since $\overline C_-$ and $\overline C_+$ are disjoint. To see that the infimum is attained, let $\{E_j\}$ be a minimizing sequence. Since $M$ is compact,
$e^{-f}$ is bounded above and below by positive constants. Moreover, $|\beta|\le1$ and $0\le\operatorname{Tr}_S\chi_{E_j}\le1$, so
$$
 \int_{\partial^*E_j\cap M^\circ}e^{-f}
 \le
 \cF_{f,0,\beta}(E_j)+\int_Se^{-f} .
$$
Thus $\{E_j\}$ has uniformly bounded perimeter. By compactness for
sets of finite perimeter, after passing to a subsequence,
$$
 \chi_{E_j}\longrightarrow\chi_E
 \qquad\text{in }L^1(M)
$$
for some finite perimeter set $E$. The constraints
\eqref{eq:polar-set-constraints} pass to the limit. Finally, the standard lower semicontinuity of the capillary perimeter functional gives
$$
 \cF_{f,0,\beta}(E)
 \le
 \liminf_{j\to\infty}\cF_{f,0,\beta}(E_j).
$$
Hence $E$ minimizes $\cF_{f,0,\beta}$ among all sets of finite perimeter satisfying \eqref{eq:polar-set-constraints}.
\medskip\\
\noindent{\bf Case (2): $M$ is noncompact.} 
Choose a connected compact set $K\subset M$ containing a neighborhood
of
$$
 S\cup\overline V_-\cup\overline V_+.
$$
By Lemma~\ref{lem:exhaustion}, there exists a smooth proper function $\sigma:M\rightarrow[0,\infty),$ with
$$
 |d\sigma|\le1,
 \qquad
 \sigma=0
 \quad\text{on a neighborhood of }K.
$$
Let $\zeta:\mathbb R\to[0,1]$ be a smooth nondecreasing function, with
$$
 \zeta=0\quad\text{on }(-\infty,1],
 \qquad
 \zeta=1\quad\text{on }[2,\infty).
$$
For a sufficiently large regular value $T>3$ of $\sigma$, set
\begin{equation}\label{eq:mu-barrier}
 \eta_T(s)=\frac{8\zeta(s)}{T-s},
 \qquad
 h_T=-\eta_T(\sigma)
 \quad\text{on }\{\sigma<T\}.
\end{equation}
Since $\sigma=0$ near $K$, we have
$$
 h_T=0
 \quad\text{on a neighborhood of }
 S\cup\overline V_-\cup\overline V_+.
$$
Moreover,
$$
 \eta_T'=
 \frac{8\zeta'}{T-s}
 +\frac{8\zeta}{(T-s)^2}
 =
 \frac{\eta_T^2}{8}+e_T,
$$
where 
\begin{equation} \label{eq:riccati-error}
 e_T=
 \frac{8\zeta'}{T-s}
 +\frac{8\zeta(1-\zeta)}{(T-s)^2}
 \ge0.   
\end{equation}
The function $e_T$ is supported in $[1,2]$ and satisfies
$$
 \|e_T\|_\infty
 \le
 \frac{8\|\zeta'\|_\infty}{T-2}
 +\frac{2}{(T-2)^2}
 \longrightarrow0
 \qquad\text{as }T\to\infty.
$$
Since $|d\sigma|\le1$,
$$
 |dh_T|
 \le
 \eta_T'(\sigma)
 \le
 \frac18h_T^2+e_T(\sigma).
$$
The set $\sigma^{-1}([1,2])$ is compact, so
$$
 \min_{\sigma^{-1}([1,2])}S_n(g, f)>0.
$$
Increasing $T$ if necessary, we may assume
$$
 2\|e_T\|_\infty
 <
 \frac12\min_{\sigma^{-1}([1,2])}S_n(g, f).
$$
It follows that
\begin{equation}\label{eq:shielding-bulk}
 S_n(g, f)+\frac12h_T^2-2|dh_T|
 \ge
 S_n(g, f)+\frac14h_T^2-2e_T(\sigma)>0
 \qquad\text{on }\{\sigma<T\}.
\end{equation}

Since $T$ is a regular value and $\sigma$ is proper, there exists
$\eps_T>0$, with $T-\eps_T>2$, such that
$$
 B_T:=\sigma^{-1}([T-\eps_T,T+\eps_T])
$$
is compact and $d\sigma\ne0$ throughout $B_T$. Thus
$$
 V:=\frac{\nabla\sigma}{|\nabla\sigma|}
$$
is smooth near $B_T$. Set
$$
 \diver_fV=\diver V-\langle\nabla f,V\rangle.
$$
Then $\diver_fV$ is continuous
on $B_T$, and therefore
$$
L_T:=\max_{B_T}|\diver_fV|<\infty.
$$
Choose regular values
$$
 T-\eps_T<s_0<s_a<s_1<T
$$
so close to $T$ that
$$
 \frac8{T-s_0}>L_T+1.
$$
Since $\zeta=1$ for $\sigma>2$, we have
$$
 h_T=-\frac8{T-\sigma}
 \qquad\text{on }\{2<\sigma<T\},
$$
and therefore
\begin{equation}\label{eq:terminal-mu}
 h_T<-|\diver_fV|-1
 \qquad\text{on }\{s_0<\sigma<s_1\}.
\end{equation}
We now extend $h_T$ to a global smooth prescribed function. Choose $\widehat s\in(s_1, T)$ and a smooth function
$\chi:\mathbb R\to[0,1]$ such that
$$
 \chi=1\quad\text{on }(-\infty,s_1],
 \qquad
 \chi=0\quad\text{on }[\widehat s,\infty).
$$
Define
$$
 h=
 \begin{cases}
 -\chi(\sigma)\eta_T(\sigma),&\sigma<T,\\
 0,&\sigma\ge T.
 \end{cases}
$$
Then $h\in C^\infty(M)$ and
$$
 h=h_T\quad\text{on }\{\sigma\le s_1\},
 \qquad
 h=0\quad\text{on a neighborhood of }
 S\cup V_-\cup V_+.
$$

Let $M_1$ be the closure of the connected component of $\{\sigma<s_1\}$ containing $K$. Then $M_1$ is compact. We minimize $\cF_{f,h,\beta}$ among sets of finite perimeter $E\subset M_1$
satisfying \eqref{eq:polar-set-constraints} and the additional constraint
\begin{equation}\label{eq:outer-set-constraint}
 \chi_E=0
 \quad\text{a.e. on }
 M_1\cap\{s_a<\sigma<s_1\}.
\end{equation}
We regard each competitor as a subset of $M$ by setting $\chi_E=0$ on $M\setminus M_1$. Since
\eqref{eq:outer-set-constraint} makes $E$ empty on a neighborhood of the artificial boundary $\{\sigma=s_1\}$, this extension creates no additional reduced boundary there. The constrained class is nonempty. Indeed, since
$\sigma=0$ on a neighborhood of $V_-\cup V_+$ and $s_a>0$, the collar region $V_-$ is disjoint from both $C_+$ and $M_1\cap\{s_a<\sigma<s_1\}$. Hence $V_-$ is a finite perimeter competitor satisfying
\eqref{eq:polar-set-constraints} and
\eqref{eq:outer-set-constraint}.

Let $\{E_j\}$ be a minimizing sequence in this class and fix one competitor $V_-$. We may assume
$$
 \cF_{f,h,\beta}(E_j)
 \le
 \cF_{f,h,\beta}(V_-)+1.
$$
Since $M_1$ is compact, $e^{-f}$ is bounded above and below there by
positive constants, and $h$ is bounded. Since $|\beta|\le1$, we have
$$
 \int_{\partial^*E_j\cap M_1^\circ}e^{-f}
 \le
 \cF_{f,h,\beta}(E_j)
 +\int_S e^{-f}
 +\int_{M_1}|h|e^{-f}.
$$
Hence the weighted perimeters, and therefore the ordinary perimeters, of $E_j$ are uniformly bounded. Also
$$
 \|\chi_{E_j}\|_{L^1(M_1)}
 \le \Vol(M_1).
$$
By compactness for sets of finite perimeter, after passing to a subsequence,
$$
 \chi_{E_j}\longrightarrow\chi_E
 \qquad\text{in }L^1(M_1)
$$
for some finite perimeter set $E\subset M_1$. The constraints \eqref{eq:polar-set-constraints} and
\eqref{eq:outer-set-constraint} pass to the limit. By the standard lower semicontinuity of the capillary functional,
$$
 \cF_{f,h,\beta}(E)
 \le
 \liminf_{j\to\infty}\cF_{f,h,\beta}(E_j).
$$
Hence $E$ minimizes $\cF_{f,h,\beta}$ among all sets of finite perimeter in $M_1$ satisfying \eqref{eq:polar-set-constraints} and
\eqref{eq:outer-set-constraint}.
\medskip\\
\noindent{\bf Inactivity of the pole constraints.}
Recall that the collars $C_\pm\Subset V_\pm$, the interior faces $\Gamma_\pm$ of $V_\pm$, and the vector fields $Q_\pm$ were chosen above so that \eqref{eq:pole-Q-divergence} and \eqref{eq:pole-wall-signs} hold.

Consider first the $p_-$ collar and set $W_-=V_-\setminus E.$
If $|W_-|=0$, then
\begin{equation}\label{p-inact}
\chi_E=1\qquad\text{a.e. on }V_-,   
\end{equation}
and there is nothing to prove. Suppose therefore that $|W_-|>0$ and set $E'=E\cup V_-.$ Then $E'\setminus E=W_-.$ Moreover, $E'$ is an admissible competitor. Indeed, it still satisfies the two pole constraints, and in Case \textup{(2)},
$$
 V_-\subset\{\sigma=0\}
 \quad\text{is disjoint from}\quad
 \{s_a<\sigma<s_1\},
$$
so the outer constraint \eqref{eq:outer-set-constraint} is preserved.

Note
$$E'=E\quad\text{on }M\setminus V_-,
\qquad
E'=V_-\quad\text{on }V_-.$$
Thus the modification replaces the portion of $E$ inside $V_-$ by the whole collar $V_-$, while leaving $E$ unchanged outside $V_-$. The weighted perimeter comparison gives
\begin{equation}\label{eq:pole-perimeter-comparison}
\int_{\partial^*E'\cap M^\circ}e^{-f}
 -
 \int_{\partial^*E\cap M^\circ}e^{-f}
\le
\int_{\Gamma_-}e^{-f}\operatorname{Tr}_{V_-}\chi_{W_-}
 -
 \int_{\partial^*E\cap V_-}e^{-f}.
\end{equation}
Applying the finite perimeter Gauss--Green formula to $W_-$ with the vector field $e^{-f}Q_-$ gives
\begin{equation}\label{eq:pole-gauss-green}
 \int_{W_-}e^{-f}\diver_fQ_-
 =
 \int_{\Gamma_-}
 e^{-f}\operatorname{Tr}_{V_-}\chi_{W_-}
 -
 \int_{\partial^*E\cap V_-}
 e^{-f}\langle Q_-,N\rangle
 +
 \int_S
 e^{-f}\operatorname{Tr}_S\chi_{W_-}
 \langle Q_-,\nu\rangle,
\end{equation}
Here $N$ denotes the measure theoretic outward unit normal of $E$. Combining \eqref{eq:pole-perimeter-comparison}, \eqref{eq:pole-gauss-green} and $|\langle Q_-, N\rangle|\leq1$, we obtain
\begin{equation}\label{eq:pole-perimeter-final}
 \int_{\partial^*E'\cap M^\circ}e^{-f}
 -
 \int_{\partial^*E\cap M^\circ}e^{-f}
 \le
 \int_{W_-}e^{-f}\diver_fQ_-
 -
 \int_S
 e^{-f}\operatorname{Tr}_S\chi_{W_-}
 \langle Q_-,\nu\rangle .
\end{equation}
Since $h=0$ on $V_-$, the volume term $\int_Mh\,\chi_Ee^{-f}$ does not change. The change of
the wall term is
$$
 -\int_S
 \beta\,\operatorname{Tr}_S\chi_{W_-}\,e^{-f}.
$$
Therefore
\begin{equation} \label{eq:pole-minus-comparison}
\cF_{f,h,\beta}(E')-\cF_{f,h,\beta}(E)
\le
\int_{W_-}e^{-f}\diver_fQ_-
-
\int_S
 e^{-f}\operatorname{Tr}_S\chi_{W_-}
 \bigl(\langle Q_-,\nu\rangle+\beta\bigr).
\end{equation}
Note $V_-\cap S=O_-$. Since the trace $\operatorname{Tr}_S\chi_{W_-}$ vanishes almost everywhere on $S\setminus O_-$ and is nonnegative almost
everywhere on $O_-$, \eqref{eq:pole-wall-signs} implies
$$
 \int_S e^{-f}\operatorname{Tr}_S\chi_{W_-}
 \bigl(\langle Q_-,\nu\rangle+\beta\bigr)\ge0.
$$
Together with \eqref{eq:pole-Q-divergence}, \eqref{eq:pole-minus-comparison} therefore gives
$$
 \cF_{f,h,\beta}(E')-\cF_{f,h,\beta}(E)
 \le
 -\frac{\kappa}{2}\int_{W_-}e^{-f}<0,
$$
where the strict inequality follows from the assumption
$|W_-|>0$. This contradicts the minimality of $E$.
Hence $|W_-|=0$, and we have \eqref{p-inact}.

The $p_+$ collar follows from the same comparison applied to the
complement. Let
$$
 \Omega=
 \begin{cases}
 M,&\text{in Case~\textup{(1)}},\\
 M_1,&\text{in Case~\textup{(2)}},
 \end{cases}
 \qquad
 E^c=\Omega\setminus E.
$$
In this paragraph, all perimeters and volume terms are taken relative to $\Omega$. Since relative perimeter is invariant under complementation and $\chi_E=1-\chi_{E^c}$, we have
$$
 \cF_{f,h,\beta}(E)
 =
 \cF_{f,-h,-\beta}(E^c)
 -\int_S\beta e^{-f}
 -\int_\Omega h e^{-f}.
$$
The last two terms are independent of $E^c$. Hence $E^c$ minimizes $\cF_{f,-h,-\beta}$ in the corresponding constrained class. Since the second inequality in \eqref{eq:pole-wall-signs} is exactly $\langle Q_+,\nu\rangle+(-\beta)\ge0$, applying the preceding $p_-$ comparison to $E^c$ gives
$$
 \chi_{E^c}=1\qquad\text{a.e. on }V_+,
$$
and hence
\begin{equation}\label{p+inact}
 \chi_E=0\qquad\text{a.e. on }V_+.
\end{equation}

Since $C_\pm\Subset V_\pm$ and $\chi_E$ is constant on $V_\pm$, the pole constraints \eqref{eq:polar-set-constraints} impose no restriction near the capillary interface.
\medskip\\
\noindent{\bf Inactivity of the outer constraint.}
This is needed only in Case~\textup{(2)}. Choose
$$
 s_0<b_0<b_1<s_a.
$$
For almost every $s\in(b_0,b_1)$, set
$$
 U_s=M_1\cap\{s<\sigma<s_1\},
 \qquad
 W_s=E\cap U_s.
$$
Since $\sigma=0$ on a neighborhood of $S\cup\overline V_-\cup\overline V_+$,
$$
 \overline{U_s}\cap(S\cup V_-\cup V_+)=\varnothing.
$$
Moreover, \eqref{eq:outer-set-constraint} implies that the trace of $\chi_E$ vanishes near $\{\sigma=s_1\}$.

Recall that $V=\nabla\sigma/|\nabla\sigma|$ on
$\{s_0<\sigma<s_1\}$, pointing toward increasing $\sigma$. Applying the finite perimeter Gauss--Green formula to $W_s$ with the
vector field $e^{-f}V$ gives
\begin{equation}\label{eq:cut-gauss-green}
 \int_{W_s}e^{-f}\diver_fV
 =
 \int_{\partial^*E\cap U_s}
 e^{-f}\langle V,N\rangle
 -
 \int_{E\cap\{\sigma=s\}}e^{-f}.
\end{equation}
Therefore,
\begin{equation}\label{eq:cut-area-comparison}
 \int_{E\cap\{\sigma=s\}}e^{-f}
 -
 \int_{\partial^*E\cap U_s}e^{-f}
 \le
 -\int_{W_s}e^{-f}\diver_fV.
\end{equation}
The competitor $E\setminus W_s$ is admissible. Indeed, since $\overline{U_s}\cap(S\cup V_-\cup V_+)=\varnothing$, the modification does not affect either the pole constraints or the boundary term over $S$, while deletion clearly preserves
\eqref{eq:outer-set-constraint}. Hence minimality and \eqref{eq:cut-area-comparison} give
\begin{equation}\label{eq:terminal-comparison}
 0
 \le
 \cF_{f,h,\beta}(E\setminus W_s)
 -
 \cF_{f,h,\beta}(E)
 \le
 \int_{W_s}e^{-f}(h-\diver_fV)
 \le
 -\int_{W_s}e^{-f},
\end{equation}
where the last inequality follows from
\eqref{eq:terminal-mu}. Thus $|W_s|=0$.

Choosing such an $s<b_1$, we obtain
$$
 E\cap\{b_1<\sigma<s_1\}=\varnothing
 \qquad\text{a.e.}
$$
Since $b_1<s_a$, this empty collar strictly contains the one imposed by \eqref{eq:outer-set-constraint}. Hence the outer constraint is inactive
near the capillary interface.

In Case \textup{(2)}, we regard $E\subset M_1$ as a subset of $M$. Since $E$ vanishes in a neighborhood of $\{\sigma=s_1\}$, no reduced boundary is created along this artificial boundary. Since all auxiliary constraints are separated from $Y$, there exists an open neighborhood $\mathcal U$ of $Y$ such that
$$
 \cF_{f,h,\beta}(E)\le \cF_{f,h,\beta}(E')
$$
for every set of finite perimeter $E'$ with
$E'\triangle E\Subset\mathcal U.$
This proves item \ref{itemi}.
\medskip

\begin{figure}[!htbp]
\centering
\begin{tikzpicture}[x=1cm,y=1cm,font=\small,
  >=Latex,line cap=round,line join=round]
\fill[figblue!7] (0,0)--(6,0)
  .. controls (5.40,1.20) and (5.80,2.60) .. (6.55,3.65)
  --(0,3.65)--cycle;

\draw[densely dashed,black!30]
  (0,0)--(0,3.65)--(12,3.65)--(12,0);

\node[figblue] at (2.55,2.65) {$E$};
\node[black!65] at (9.35,2.65) {$M\setminus E$};
\node[black!65] at (11.45,3.25) {$M$};

\filldraw[fill=figblue!17,draw=figblue,line width=.8pt]
  (1,0)--(1,.45)
  .. controls (1,1.65) and (4,1.65) .. (4,.45)
  --(4,0)--cycle;

\filldraw[fill=figorange!13,draw=figorange,line width=.8pt]
  (8,0)--(8,.45)
  .. controls (8,1.65) and (11,1.65) .. (11,.45)
  --(11,0)--cycle;

\node[figblue,align=center] at (2.5,.65)
  {$C_-$\\[3pt]$\chi_E=1$};
\node[figorange,align=center] at (9.5,.65)
  {$C_+$\\[3pt]$\chi_E=0$};

\draw[figteal,line width=1.5pt] (6,0)
  .. controls (5.40,1.20) and (5.80,2.60) .. (6.55,3.65);
\node[figteal,anchor=west] at (6.05,2.05) {$Y$};

\draw[black!65,line width=.7pt] (0,0)--(12,0);
\draw[figblue,line width=1.4pt] (0,0)--(6,0);
\node[figblue] at (.42,.27) {$U$};
\node[anchor=east,black!65] at (12,-.91) {$S=\partial M$};

\draw[figblue,line width=3pt] (2.12,0)--(2.88,0);
\draw[figorange,line width=3pt] (9.12,0)--(9.88,0);

\node[figblue,anchor=north] at (2.5,-.18)
  {$(P\circ\Phi)^{-1}(p_-)$};
\node[figorange,anchor=north] at (9.5,-.18)
  {$(P\circ\Phi)^{-1}(p_+)$};

\fill[figteal] (6,0) circle (2.2pt);
\draw[figteal!75] (6,-.07)--(6,-.65);
\node[figteal,align=center,anchor=north] at (6,-.72)
  {regular contact boundary\\[2pt]
   $\partial Y_{\rm reg}$};

\filldraw[fill=black!2,draw=black!60,line width=.85pt]
  (6,-3.45) circle (1.12);
\draw[densely dashed,black!30]
  (6,-3.45) ellipse (.28 and 1.12);
\node[fill=black!2,inner sep=3pt] at (6,-3.45)
  {$\mathbb S^{n-1}$};

\draw[->,figblue,line width=1pt]
  (2.5,-.83)
  .. controls (2.5,-2.85) and (3.50,-3.45) .. (4.80,-3.45);
\draw[->,figorange,line width=1pt]
  (9.5,-.83)
  .. controls (9.5,-2.85) and (8.50,-3.45) .. (7.20,-3.45);

\node[figblue,fill=white,inner sep=2pt] at (2.8,-2.35) {$P\circ\Phi$};
\node[figorange,fill=white,inner sep=2pt] at (9.2,-2.35) {$P\circ\Phi$};

\fill[figblue] (4.88,-3.45) circle (2pt);
\fill[figorange] (7.12,-3.45) circle (2pt);

\node[figblue,anchor=south east] at (4.85,-3.28) {$p_-$};
\node[figorange,anchor=south west] at (7.15,-3.28) {$p_+$};

\end{tikzpicture}

\caption{The pole preimages and the prescribed inclusion and exclusion regions. The set $C_-$ is prescribed to lie in $E$, whereas $C_+$ is prescribed to lie outside $E$. The wetted region $U$ separates the two pole preimages. The precise identification of its regular boundary with the contact boundary of $Y$ is made in item \ref{itemiii}.}
\label{fig:polar-neighborhoods}
\end{figure}
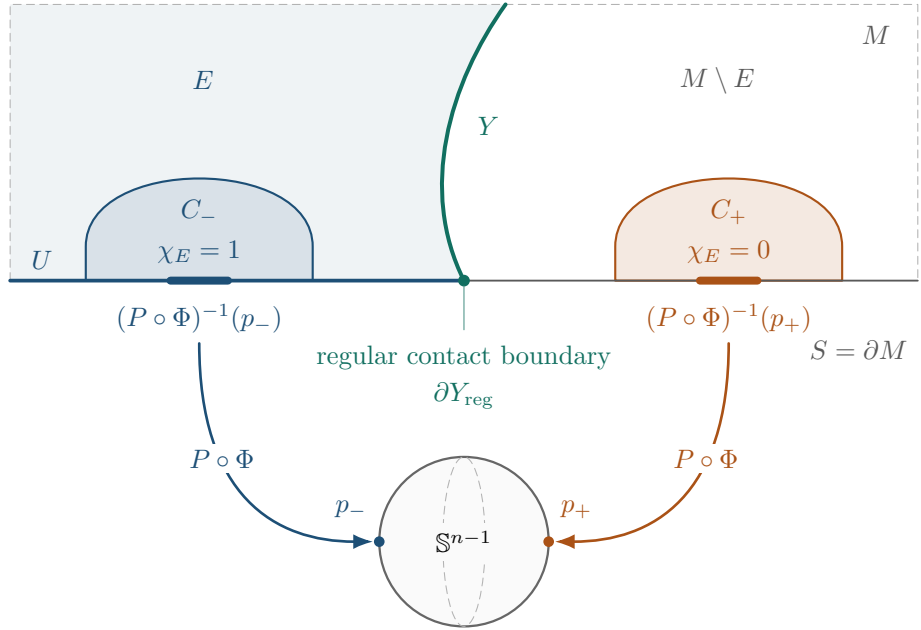

\noindent{\bf Proof of item \ref{itemii}.}
In Case \textup{(2)}, the preceding cutting argument gives $Y\subset\{\sigma\le b_1\},$
and hence $Y$ is compact since $\sigma$ is proper. In Case \textup{(1)}, compactness is immediate. On $Y_{\rm reg}$, the measure theoretic outward normal of $E$ agrees with the smooth unit normal to the regular interface. Thus $Y_{\rm reg}$ is two-sided and is naturally oriented by the normal
pointing out of $E$.

By the preceding argument,
$$
 \chi_E=1 \quad\text{a.e. on }V_-,
 \qquad
 \chi_E=0 \quad\text{a.e. on }V_+.
$$
Since $V_\pm$ are relatively open in $M$, this implies
$Y\cap V_\pm=\varnothing$. Together with
$$
 (P\circ\Phi)^{-1}(p_\pm)
 \subset \operatorname{int}_S(C_\pm\cap S)
 \subset \operatorname{int}_S(V_\pm\cap S),
$$
we obtain
$$
 (Y\cap S)\cap
 (P\circ\Phi)^{-1}(\{p_-,p_+\})=\varnothing.
$$
Thus the prescribed angle $\theta$ takes values in $(0,\pi)$
on the compact set $Y\cap S$. By continuity, there exists
$\theta_0>0$ such that
$$
 \theta_0\le\theta\le\pi-\theta_0
 \qquad\text{on }Y\cap S.
$$
After decreasing $\theta_0$ if necessary, the same bounds hold
on a neighborhood of $Y\cap S$ in $S$.

Together with the local minimality established in
item~\ref{itemi}, this uniform nondegeneracy allows us to apply
the local capillary regularity in
\cite[Lemma 4.4(i)]{WWXZ}.
It follows that $Y_{\rm reg}$ is a smooth embedded hypersurface
with boundary
$$
 \partial Y_{\rm reg}=Y_{\rm reg}\cap S.
$$
Finally, since
$\overline{\partial Y_{\rm reg}}\subset Y\cap S$,
the angle bound above also holds on
$\overline{\partial Y_{\rm reg}}$.
This proves item \ref{itemii}.

\medskip
\noindent{\bf Proof of item \ref{itemiii}.}
We first show that the wetted region on $S$, namely the subset of $S$
where the $BV$ trace $\operatorname{Tr}_S\chi_E$ equals one, has finite perimeter in $S$.
The proof of item~\ref{itemii} shows that $|\beta|<1$ on a
neighborhood of the compact set $Y\cap S$.
By item~\ref{itemi}, $E$ is locally minimizing near $Y$,
and $h=0$ near $S$.
In smooth boundary charts, the surface energy has a regular
elliptic integrand, and $|\beta|<1$ gives the strict adhesion condition in \cite[Theorem 1.10]{DPM}.
That theorem therefore shows that the wetted region on $S$ has locally finite perimeter in $S$ near $Y\cap S$. Moreover, its boundary
regularity shows that $Y\cap S$ is locally a $C^{1,1/2}$
submanifold of dimension $n-2$ away from a set of zero
$\mathcal H^{n-2}$-measure. Consequently,
$$
 \mathcal H^{n-1}(Y\cap S)=0.
$$

Away from $Y\cap S$, the boundary trace
$\operatorname{Tr}_S\chi_E$ is locally constant, taking the
value zero or one almost everywhere. Let $U$ be the union
of all relatively open subsets of $S\setminus(Y\cap S)$
on which this trace equals one almost everywhere.
Then $U$ is relatively open in $S$, and
$$
 \partial_S U\subset Y\cap S.
$$
Since $\mathcal H^{n-1}(Y\cap S)=0$, we have
$$
 \operatorname{Tr}_S\chi_E=\chi_U
 \qquad\text{a.e. on }S.
$$
The local finiteness of the perimeter established above
near $Y\cap S$, together with local constancy of $\chi_U$
elsewhere, shows that $U$ has locally finite perimeter
throughout $S$. Compactness of $S$ therefore implies
that $U$ has finite perimeter in $S$.

The pole comparisons give $V_-\cap S\subset U$ and
$(V_+\cap S)\cap U=\varnothing$.
Since $V_+\cap S$ is relatively open in $S$, it is also
disjoint from $\overline U$.
The pole preimages $(P\circ\Phi)^{-1}(p_\pm)$ are compact
and contained in $V_\pm\cap S$, respectively. Hence
$$
 (P\circ\Phi)^{-1}(p_-)\Subset U,
 \qquad
 \overline U\cap(P\circ\Phi)^{-1}(p_+)=\varnothing.
$$

Set
$$
 \Sigma_Y=Y\setminus Y_{\rm reg},
 \qquad
 K_\partial=\partial_S U\cap\Sigma_Y.
$$
By \cite[Lemma 4.4(i)]{WWXZ}, $Y_{\rm reg}$ is a smooth
embedded hypersurface with
$\partial Y_{\rm reg}=Y_{\rm reg}\cap S$.
Near each regular contact point, the nondegenerate contact angle ensures that $Y_{\rm reg}$ meets $S$ transversely.
The boundary trace equals one on one side of
$\partial Y_{\rm reg}$ in $S$ and zero on the other.
Together with $\partial_S U\subset Y\cap S$, this gives
$$
 \partial_S U\setminus K_\partial=\partial Y_{\rm reg}.
$$
Thus $\partial_S U\setminus K_\partial$ is a smooth embedded hypersurface of $S$, and $U$ locally lies on one side of it.

Since $\Sigma_Y$ is closed in $Y$ and $\partial_S U$
is closed in $S$, the set $K_\partial$ is a closed subset
of the compact set $Y$, and hence is compact.
For $n\le4$,  \cite[Lemma 4.4(i)]{WWXZ} gives
$\Sigma_Y=\varnothing$, so $K_\partial=\varnothing$.
For $n\ge5$, it gives
$$
 \mathcal H^{n-5+\varepsilon}(\Sigma_Y)=0
 \qquad\text{for every }\varepsilon>0,
$$
and therefore the Hausdorff dimension of $K_\partial$ satisfies
$$
 \dim K_\partial
 \le \dim \Sigma_Y
 \le n-5<n-2.
$$

Let
$$
 P_{\rm lon}:
 \mathbb S^{n-1}\setminus\{p_-,p_+\}
 \longrightarrow\mathbb S^{n-2}
$$
denote the longitude projection.
The pole separation above shows that $\partial_S U$
avoids $(P\circ\Phi)^{-1}(\{p_-,p_+\})$.
Thus $G=P_{\rm lon}\circ P\circ\Phi$ is smooth on a
neighborhood of $\partial_S U$ in $S$, and in particular
on a neighborhood of $\overline{\partial Y_{\rm reg}}$,
since
$
 \overline{\partial Y_{\rm reg}}\subset\partial_S U.
$
Applying Lemma~\ref{lem:longitude-degree} to $P\circ\Phi$
and $U$, and using $\deg P=1$, we obtain
$$
 |\deg G|=|\deg(P\circ\Phi)|=|\deg\Phi|\ne0.
$$
This proves item \ref{itemiii}.
\medskip

\noindent{\bf Proof of item \ref{itemiv}.}
In Case \textup{(2)} in the proof of item \ref{itemi}, since
$$
 Y\subset\{\sigma\le b_1\},
 \qquad b_1<s_1,
$$
we may choose a compact neighborhood $\mathcal K$ of $Y$ such that
$$
 \mathcal K\Subset\{\sigma<s_1\}.
$$
On $\mathcal K$ we have $h=h_T$, and
\eqref{eq:shielding-bulk} gives
$$
 c_0:=
 \frac12\min_{\mathcal K}
 \left(
  S_n(g,f)+\frac12h^2-2|dh|
 \right)>0.
$$
In Case \textup{(1)} in the proof of item \ref{itemi}, $h=0$, and we instead set
$$
 c_0:=\frac12\min_M S_n(g,f)>0.
$$
Since $N$ is the outward unit normal of $E$ along $Y_{\rm reg}$, $Nh=dh(N)\ge-|dh|,$
and hence, in either case,
$$
 \frac12
 \left(
  S_n(g,f)+2Nh+\frac12h^2
 \right)
 \ge c_0
 \qquad\text{on }Y_{\rm reg}.
$$
By the choice of $P$ at the beginning of the proof, we may set
$$
 c_1=\min_S(H_f-D_{P\circ\Phi})\ge\frac{\delta_S}{2}>0.
$$
Since $h=0$ near $S$ and $0<\sin\theta\le1$ on
$\partial Y_{\rm reg}$, we have
$$
 \frac{H_f-D_{P\circ\Phi}+h\cos\theta}{\sin\theta}
 =\frac{H_f-D_{P\circ\Phi}}{\sin\theta}
 \ge c_1 \qquad\text{on }\partial Y_{\rm reg}.
$$
The local minimality in item \ref{itemi} gives the first variation
identities and the stability inequality used in Lemma \ref{lem:variation}.
Applying that lemma with $P\circ\Phi$ in place of its boundary map
$\Phi$, and with the corresponding $\theta$ and $G$, therefore gives
$$
 \Q_Y(v)\ge
 c_0\int_{Y_{\rm reg}}e^{-f_Y}v^2
 +
 c_1\int_{\partial Y_{\rm reg}}e^{-f_Y}v^2
$$
for every $v\in C_c^\infty(Y_{\rm reg})$.
This proves item \ref{itemiv} and completes the proof.
\end{proof}
\medskip
Now we are ready to prove Theorem \ref{thm:obstruction}.
\begin{proof}[Proof of Theorem \ref{thm:obstruction}]
Suppose, to the contrary, that there exist
$f\in C^\infty(M)$ and $\Phi:S\to\SSS^{n-1}$ satisfying \eqref{eq:obstruction-data}.

Set $m=n-1$, and let $P,\,E,\,h,\,Y$ be supplied by
Lemma \ref{lem:shielding}. Let $G$ be the longitude map of
$P\circ\Phi$ on $\partial Y_{\rm reg}$. By item \ref{itemiii} of Lemma \ref{lem:shielding},
$$
 |\deg G|=|\deg\Phi|\ne0.
$$
Fix
$$
 0<\delta\le \min\{c_0,c_1\},
$$
where $c_0,c_1$ are given by \eqref{eq:slice-margins}.

If $m=2$, then $Y$ is smooth by \cite[Lemma 4.4(i)]{WWXZ} (also see the proof of
Lemma~\ref{lem:shielding}\ref{itemiii}). Since $\deg G\ne0$, at least one connected
component $Y_0$ of $Y$ satisfies
$$
 \deg\left(G|_{\partial Y_0}\right)\ne0.
$$
Restricting \eqref{eq:slice-margins} to $Y_0$, the above choice of $\delta$ gives the spectral inequality in \cite[Lemma 3.4]{WWXZ}, with weight $-f_Y$ and coefficient $\mu=a_n>\frac12$. Hence
\cite[Lemma 3.4]{WWXZ} gives a contradiction.

Assume now that $m\ge3$. Since $m=n-1$,
$$
 a_n=\frac{m+1}{m+2}>\frac{m}{m+1}.
$$
Choose $\mu$ in the nonempty interval
$$
 \frac{m}{m+1}<\mu\le
 \min\left\{
 a_n,\,
 \frac{m}{m+1}+\frac{9}{(m+1)(m-2)^2}
 \right\},
$$
in accordance with \cite[(4.21)]{WWXZ}. Since $\mu\le a_n$,
$$
 -\mu|df_Y|_{g_Y}^2
 \ge
 -a_n|df_Y|_{g_Y}^2,
$$
so replacing $a_n$ by $\mu$ only enlarges the quadratic form. Therefore \eqref{eq:slice-margins}, with the same $\delta$ fixed above,
gives, for every $v\in C_c^\infty(Y_{\rm reg})$,
\begin{equation}\label{eq:spectral-original}
\begin{aligned}
0\le{}&
\int_{Y_{\rm reg}}e^{-f_Y}
\left[
 |dv|_{g_Y}^2+
 \left(
 \frac12\bigl(
 \Sc_{g_Y}+2\Delta_{g_Y}f_Y
 -\mu|df_Y|_{g_Y}^2
 \bigr)-\frac{\delta}{4}
 \right)v^2
\right]
\\
&+
\int_{\partial Y_{\rm reg}}e^{-f_Y}
\left(
 H_{\partial Y}^{f_Y}-D_G-\delta
\right)v^2.
\end{aligned}
\end{equation}
This is precisely the spectral inequality in
\cite[Lemma 4.6]{WWXZ}, where their function $f$
corresponds to $-f_Y$ in our notation, with
$\rho=e^{-f_Y}$ and $\widehat F=G$.

By items \ref{itemi} and \ref{itemii} of
Lemma \ref{lem:shielding}, the interface $Y$ of $E$ is compact. The auxiliary constraints are inactive
near $Y$, so $E$ is locally minimizing there.
Moreover, the contact angle is uniformly bounded away
from $0$ and $\pi$. This enables the local volume and singular set estimates in
\cite[Sections 4.3 and 5]{WWXZ} to be applied to $Y$. Together with $\deg G\ne0$ and the spectral inequality above, these properties allow us to apply the dimension reduction construction of \cite[Sections 4.3--4.7]{WWXZ}.

If $Y=Y_{\rm reg}$, take $W\equiv1$; otherwise \cite[Lemma 4.11]{WWXZ} gives a smooth positive function
$W\in C^\infty(Y_{\rm reg})$. In either case,
$$
 \widetilde g:=W^{\frac43}g_Y
$$
is complete on $Y_{\rm reg}$.

If $Y=Y_{\rm reg}$, choose antipodal regular values
$q_\pm$ of $G$. Since $D_G$ is bounded on the compact
manifold $\partial Y$, choose $\varepsilon>0$ such that
$\varepsilon\max_{\partial Y}D_G\le\delta/2$.
Using the construction of Lemma \ref{lem:spherical-tools},
choose a smooth degree-one map
$P_G:\SSS^{m-1}\to\SSS^{m-1}$
with $\operatorname{Lip}P_G\le1+\varepsilon$ such that
$P_G^{-1}(q_\pm)$ are small closed spherical caps
centered at $q_\pm$, each contained in a neighborhood
consisting entirely of regular values of $G$.
Set $\widetilde G=P_G\circ G$. Then
$$
 D_{\widetilde G}
 \le(1+\varepsilon)D_G
 \le D_G+\frac{\delta}{2}
 \qquad\text{on }\partial Y.
$$

If $Y\ne Y_{\rm reg}$, let $U$ and $K_\partial$ be as in
the proof of item \ref{itemiii} of Lemma \ref{lem:shielding}.
The map $G$ extends smoothly to a neighborhood of the compact
set $\partial_S U$ in $S$.
Together with $\dim K_\partial<m-1$, this gives
$\mathcal H^{m-1}(G(K_\partial))=0$.
Choose a regular value $q_-$ of $G|_{\partial Y_{\rm reg}}$
outside $G(K_\partial)$, and set $q_+=-q_-$.
Its preimage $(G|_{\partial Y_{\rm reg}})^{-1}(q_-)$ is compact and discrete, hence finite. Choose an open neighborhood $\mathcal C$ of $K_\partial$ in $S$
such that $\overline{\mathcal C}\cap (G|_{\partial Y_{\rm reg}})^{-1}(q_-)=\varnothing$.
Apply the construction of \cite[Lemma 4.12]{WWXZ} to
$G|_{\partial Y_{\rm reg}}$, modifying it only at points
of $\mathcal C\cap\partial Y_{\rm reg}$ and leaving it
unchanged on $\partial Y_{\rm reg}\setminus\mathcal C$.
The resulting smooth map
$G_1:\partial Y_{\rm reg}\to\SSS^{m-1}$ satisfies
$$
 G_1=G\quad\text{on }\partial Y_{\rm reg}\setminus\mathcal C,
 \qquad
 G_1=q_+\quad\text{on }\partial Y_{\rm reg}\setminus\mathcal K
$$
for some compact set $\mathcal K\subset\partial Y_{\rm reg}$.
In particular, $G_1=G$ on a neighborhood of $(G|_{\partial Y_{\rm reg}})^{-1}(q_-)$ in
$\partial Y_{\rm reg}$. Since the modification takes place
within $\SSS^{m-1}\setminus\{q_-\}$, it introduces no new
preimages of $q_-$. Thus
$$
 G_1^{-1}(q_-)
 =\bigl(G|_{\partial Y_{\rm reg}}\bigr)^{-1}(q_-).
$$
Since $G_1$ agrees with $G|_{\partial Y_{\rm reg}}$
on a neighborhood of these points, their local orientation
signs are also unchanged. Hence $\deg G_1=\deg G$.
Next apply the pole modification in \cite[Lemma 4.13]{WWXZ}
to obtain $\widetilde G:\partial Y_{\rm reg}\to\SSS^{m-1}$.
The modification has the form $\widetilde G=P_1\circ G_1$,
where $P_1:\SSS^{m-1}\to\SSS^{m-1}$ is a smooth degree-one
map satisfying $P_1(q_+)=q_+$. Hence
$\deg\widetilde G=\deg G_1$.
Moreover, $\widetilde G=P_1(q_+)=q_+$ wherever $G_1=q_+$,
so $\widetilde G$ remains equal to $q_+$ on $\partial Y_{\rm reg}\setminus\mathcal K$.

In either case, the metric $\widetilde g$, the
weight $e^{-f_Y}W^{-2(m-2)/3}$,  and the modified smooth map
$\widetilde G:\partial Y_{\rm reg}\to\SSS^{m-1}$ satisfy
the spectral inequality of \cite[Lemma 4.13]{WWXZ} and
$$
 |\deg\widetilde G|=|\deg G|\ne0.
$$
Here the degree is understood as in Section \ref{sec:preliminary}.
Set $B_\pm=\widetilde G^{-1}(q_\pm)$.
Then every connected component of $B_\pm$ has nonempty interior
in $\partial Y_{\rm reg}$, and
$\partial Y_{\rm reg}\setminus B_+$ is relatively compact.

Finally, let $\widetilde\eta$ and
$H_{\partial Y_{\rm reg}}^{\widetilde g}$
denote respectively the outward unit conormal and the mean curvature of
$\partial Y_{\rm reg}\subset(Y_{\rm reg},\widetilde g)$. Then, \cite[Lemma 4.14]{WWXZ} yields a smooth function
$\widetilde f\in C^\infty(Y_{\rm reg})$, a constant
$$
 \widetilde\mu>\frac{m-1}{m},
$$
and a constant $\delta'>0$ such that
$$
 \Sc_{\widetilde g}
 -2\Delta_{\widetilde g}\widetilde f
 -\widetilde\mu|d\widetilde f|_{\widetilde g}^2
 \ge\delta'
 \qquad\text{on }Y_{\rm reg},
$$
and
$$
 H_{\partial Y_{\rm reg}}^{\widetilde g}
-\|d\widetilde G\|_{\tr,\widetilde g}
+\widetilde\eta\widetilde f
\ge\delta'
 \qquad\text{on }\partial Y_{\rm reg}.
$$

Hence $(Y_{\rm reg},\widetilde g)$, together with
$\widetilde G$, $\widetilde f$, and $\widetilde\mu$, satisfies the
hypotheses of \cite[Theorem~3.3]{WWXZ} in dimension $m$. Therefore,
$
 \deg\widetilde G=0,
$
contradicting $\deg\widetilde G\ne0$. Hence no pair $(f,\Phi)$ can
satisfy \eqref{eq:obstruction-data}, and the theorem follows.
\end{proof}

\begin{corollary}\label{cor:positive-form}
Let $(M,g)$ be a smooth, oriented, connected
and complete Riemannian manifold of dimension $n\ge3$, with nonempty smooth compact boundary $S$. Let $\Phi:S\to\SSS^{n-1}$ be a smooth map with nonzero degree, and let $\xi\in C^\infty(S)$ satisfy
$$
 \xi\ge D_\Phi=\|d\Phi\|_{\tr}.
$$
There do not exist a smooth function $\varpi >0$ on $M$ and a constant $\varepsilon_0>0$ such that
\begin{equation}\label{eq:positive-form}
\begin{split}
 \mathcal E_\xi(v)
 &:=
 \int_M\left(|dv|^2+\frac12\Sc_gv^2\right)
 +\int_S(H_S-\xi)v^2
 \\
 &\ge
 \int_M\varpi  v^2
 +\varepsilon_0\int_Sv^2,
\end{split}
\end{equation}
for every $v\in C_c^\infty(M)$ smooth up to $S$ and not
necessarily vanishing on $S$.
\end{corollary}

\begin{proof}
Suppose, to the contrary, that there exist
$\varpi\in C^\infty(M)$ with $\varpi>0$ and a constant $\varepsilon_0>0$ satisfying \eqref{eq:positive-form}.

First suppose that $M$ is noncompact.  Choose smooth connected compact
domains $M_j$ containing a fixed collar of $S$, with
$$
 M_j\subset\operatorname{int}_M M_{j+1},
 \qquad \bigcup_{j=1}^{\infty}M_j=M,
 \qquad \partial M_j=S\sqcup\Sigma_j,
$$
where $\Sigma_j\subset M^\circ$ is the smooth artificial boundary disjoint from $S$. Let $\nu$ denote the outward unit normal of $S=\partial M$. On $M_j$, consider the operator
$$
 L:=-\Delta_g+\frac12\Sc_g-\frac12\varpi
$$
with the Robin condition
$$
 \nu\psi+
 \left(
 H_S-\xi-\frac{\varepsilon_0}{2}
 \right)\psi=0
 \qquad\text{on }S,
$$
and the Dirichlet condition
$$
 \psi=0
 \qquad\text{on }\Sigma_j.
$$
The associated quadratic form is
$$
\begin{aligned}
 \mathcal{V}_j(v)
 :={}
 \int_{M_j}
 \left(
 |dv|^2+
 \frac12(\Sc_g-\varpi)v^2
 \right)
 +
 \int_S
 \left(
 H_S-\xi-\frac{\varepsilon_0}{2}
 \right)v^2 ,
\end{aligned}
$$
defined on $\{v\in H^1(M_j):\operatorname{Tr}_{\Sigma_j}v=0\}$. The zero extension of such a function belongs to $H^1(M)$ and has
compact support. Smooth approximation in $H^1$, with support in a
fixed compact neighborhood of $M_j$, and continuity of the boundary
trace extend \eqref{eq:positive-form} to this zero extension. Identifying $v$ with its extension, we therefore obtain
$$
\begin{aligned}
 \mathcal{V}_j(v)
 &=
 \mathcal E_\xi(v)
 -\frac12\int_{M_j}\varpi v^2
 -\frac{\varepsilon_0}{2}\int_Sv^2
 \\
 &\ge
 \frac12\int_{M_j}\varpi v^2
 +\frac{\varepsilon_0}{2}\int_Sv^2
 \ge0.
\end{aligned}
$$
Hence the lowest eigenvalue $\ell_j$ of this mixed boundary problem is nonnegative. Indeed, $\ell_j\ge\min_{M_j}\varpi/2>0$. By the variational characterization and zero extension across the artificial Dirichlet boundary,
$$
\ell_{j+1}\le \ell_j.
$$
Since $\ell_j\ge0$, the sequence $\{\ell_j\}$ converges to a limit $\ell\ge0.$

Choose a nonnegative first eigenfunction $\psi_j$, which is strictly
positive on $M_j\setminus\Sigma_j$ by the strong maximum principle and the Hopf lemma. Normalize it by $\psi_j(o)=1$
at a fixed interior point $o\in M_1^\circ$. By the Harnack inequality, local elliptic estimates, and the boundary estimates on the fixed collar of $S$, after passing to a subsequence we obtain smooth
convergence on compact subsets of $M$, up to $S$, to a nonnegative smooth function $\psi$ with $\psi(o)=1$. It satisfies
$$
 L\psi=\ell\psi
 \qquad\text{on }M,
$$
and
$$
 \nu\psi+
 \left(
 H_S-\xi-\frac{\varepsilon_0}{2}
 \right)\psi=0
 \qquad\text{on }S.
$$
The strong maximum principle gives $\psi>0$ in $M^\circ$.
If $\psi$ vanished at a point of $S$, the Hopf lemma would
give $\nu\psi<0$ there, whereas the Robin condition would give
$\nu\psi=0$. Thus $\psi>0$ also on $S$. 

If $M$ is compact, the same conclusion follows directly from the first
Robin eigenfunction.

Set $f=-\log\psi.$ The equation $L\psi=\ell\psi$ gives
$$
 \frac{\Delta_g\psi}{\psi}
 =
 \frac12(\Sc_g-\varpi)-\ell.
$$
Recall $a_n=\frac{n}{n+1}$. Since
$$
 \Delta_gf
 =
 -\frac{\Delta_g\psi}{\psi}
 +|d\log\psi|_g^2,
$$
we obtain
$$
\begin{aligned}
 S_n(g,f)
 &=
 \Sc_g+2\Delta_gf-a_n|df|_g^2
 \\
 &=
 \varpi+2\ell
 +(2-a_n)|d\log\psi|_g^2
 >0
 \qquad\text{on }M.
\end{aligned}
$$

On the boundary, the Robin condition gives
$$
 \nu(\log\psi)
 =
 -H_S+\xi+\frac{\varepsilon_0}{2}.
$$
Hence,
$$
 H_f=
 H_S-\nu f=
 H_S+\nu(\log\psi) =
 \xi+\frac{\varepsilon_0}{2}
 >
 D_\Phi
 \qquad\text{on }S.
$$
Therefore,
$$
 S_n(g,f)>0,
 \qquad
 H_f>D_\Phi,
 \qquad
 \deg\Phi\ne0,
$$
which contradicts Theorem \ref{thm:obstruction}. This proves the corollary.
\end{proof}

\section{A proper Jang graph and its quadratic form}\label{sec:jang}
To apply the obstruction from Section \ref{sec:obstruction}, we need
a complete Jang graph with boundary isometric to $(S,g_S)$.
The following lemma constructs such a graph over a relatively open
subset of $X$ containing a collar of $S$ and shows that its intersection with $X\times[0,L]$ is compact for every $L\ge0$.

\begin{lemma}\label{lem:jang}
Let $(X^n,g)$ be a smooth, connected, compact Riemannian manifold of dimension $n\ge 3$, with nonempty smooth connected boundary $S=\partial X$. Assume that the mean curvature $H_S$ of $S$ in $(X,g)$,
computed with respect to the outward unit normal, satisfies
$H_S>n-1$.
Then there exist a connected relatively open set $\Omega\subset X$ containing a collar of $S$ and a smooth function
$u:\Omega\to[0,\infty)$ such that $u|_S=0$ and
\begin{equation}\label{eq:Jang}
 \tr_{g+du^2}
 \left(
 \frac{\Hess_g u}{\sqrt{1+|du|_g^2}}+g
 \right)=0
 \qquad\text{on }\Omega.
\end{equation}
Moreover, the graph
$$
 \Gamma=\{(x,u(x)):x\in\Omega\}
 \subset X\times[0,\infty)
$$
is a properly embedded hypersurface that is complete with respect to
its induced length metric. Its boundary is precisely
$
 \partial\Gamma=S\times\{0\}.
$
\end{lemma}

\begin{proof}
We first solve the regularized Dirichlet problems and construct uniform boundary barriers. Then, we establish local gradient estimates and uniform regularity at bounded height, which allow us to pass to a smooth limit on a relatively open set containing a collar of $S$. Finally, we prove that the limiting graph is properly embedded and
complete with respect to its induced length metric.

\medskip
\noindent\textbf{Step 1: Solve the regularized Dirichlet problems.}
For $0<\tau<1$, define
$$
 \mathcal J_\tau(v)
 =
 \diver_g\left(
 \frac{\nabla_gv}{\sqrt{1+|dv|_g^2}}
 \right)
 +n-1+\frac{1}{1+|dv|_g^2}-\tau v.
$$
By \cite[Lemma 2.2]{Eichmair} and the corrected gradient calculation needed in the fixed parameter existence argument in \cite[Appendix A, Remark A.2]{EichmairMetzger}, for every $0<\tau<1$, the assumption $H_S>n-1$
gives a smooth solution $u_\tau$ of
\begin{equation}\label{eq:jg-dirichlet}
 \mathcal J_\tau(u_\tau)=0\quad\text{in }X^\circ,
 \qquad
 u_\tau=0\quad\text{on }S.
\end{equation}
This fixed parameter existence result has no upper dimensional restriction.

The operator $\mathcal J_\tau$ is strictly elliptic, with zeroth order coefficient $-\tau<0$, and therefore satisfies the comparison principle.
Comparison with the constants $0$ and $n/\tau$ yields
\begin{equation}\label{eq:jg-height-global}
 0\le u_\tau\le\frac n\tau.
\end{equation}
If $0<\tau_1<\tau_2<1$, then
$$
 \mathcal J_{\tau_1}(u_{\tau_2})
 =(\tau_2-\tau_1)u_{\tau_2}
 \ge0
 =\mathcal J_{\tau_1}(u_{\tau_1}).
$$
Since $u_{\tau_1}=u_{\tau_2}=0$ on $S$, the comparison principle also gives
\begin{equation}\label{eq:jg-monotone}
 u_{\tau_2}\le u_{\tau_1}.
\end{equation}
Thus $u_\tau$ increases as $\tau$ decreases.

\medskip
\noindent\textbf{Step 2: Uniform boundary estimates.}
In a sufficiently small inward collar of $S$, the distance function
$s(x)=\operatorname{dist}_g(x,S)$ is smooth. Let $H_s$ denote the
mean curvature of its level hypersurfaces with respect to the unit
normal $-\nabla_g s$. Then
$$
 |\nabla_gs|_g=1,
 \qquad
 \Delta_gs=-H_s.
$$
By compactness of $S$ and continuity of $H_s$ and the strict inequality
$H_S>n-1$, there are constants $\eta_0>0$ and $\delta>0$, with $\delta$
sufficiently small, such that the collar is smooth for
$0\le s\le\delta$ and
$$
H_s\ge n-1+\eta_0
\qquad\text{for all }s\in[0,\delta].
$$

For a constant $A_0>0$, set
$$
 U(s)=A_0\log\frac{\delta}{\delta-s},
 \qquad 0\le s<\delta.
$$
Writing $a=U'=A_0/(\delta-s)$, we have $U''=a^2/A_0$.
A direct calculation gives
$$
 \mathcal J_\tau(U)
 =
 \frac{U''}{(1+a^2)^{3/2}}
 -H_s\frac{a}{\sqrt{1+a^2}}
 +n-1+\frac{1}{1+a^2}-\tau U.
$$
Since $U\ge0$, $a\ge A_0/\delta$, and
$$
 \frac{U''}{(1+a^2)^{3/2}}
 \le\frac{\delta}{A_0^2},
 \qquad
 1-\frac{a}{\sqrt{1+a^2}}
 \le\frac{1}{2a^2},
$$
we obtain
$$
 \mathcal J_\tau(U)
 \le
 -\eta_0+
 \frac{\delta+\bigl((n-1+\eta_0)/2+1\bigr)\delta^2}{A_0^2}.
$$
Fix $A_0$ sufficiently large that
$\mathcal J_\tau(U)\le-\eta_0/2$ throughout the collar for every
$0<\tau<1$. Choose $s_\tau<\delta$ with
$U(s_\tau)>n/\tau$.
Note $u_\tau\le U$ holds on both boundary hypersurfaces $\{s=0\}$ and $\{s=s_\tau\}$ of the collar
$\{0\le s\le s_\tau\}$.
By the comparison principle, we have $u_\tau\le U$ on $\{0\le s\le s_\tau\}$.
For $s_\tau\le s<\delta$, the same inequality follows
from \eqref{eq:jg-height-global} and the monotonicity of $U$.
Consequently,
\begin{equation}\label{eq:jg-collar-bounds}
 0\le u_\tau\le U(s)
 \qquad\text{for }0\le s<\delta.
\end{equation}
In particular,
\begin{equation}\label{eq:jg-collar-height}
 u_\tau\le L_c:=A_0\log2
 \qquad\text{for }0\le s\le\frac\delta 2.
\end{equation}
All tangential derivatives of $u_\tau$ vanish on $S$.
Dividing $0\le u_\tau\le U(s)$ by $s$ and letting
$s\downarrow0$ therefore gives
\begin{equation}\label{eq:jg-boundary-slope}
 0\le\partial_su_\tau|_S\le\frac{A_0}{\delta},
 \qquad
 |du_\tau|_g|_S\le\frac{A_0}{\delta}.
\end{equation}

\medskip
\noindent\textbf{Step 3: A differential inequality for $W$.}
Let $G=g+dt^2$ be the product metric on $X\times\mathbb R$.
We identify the regularized graph
$\Gamma_\tau=\operatorname{graph}(u_\tau)$ with $X$ through
the parametrization $x\mapsto(x,u_\tau(x))$, and set
$$
 \bar g=g+du_\tau\otimes du_\tau,
 \qquad
 W=\sqrt{1+|du_\tau|_g^2},
 \qquad
 w=W^{-1}.
$$
Here $\bar g$ is the induced metric on $\Gamma_\tau$.
We denote the Levi--Civita connections of $\bar g$ and $G$
by $\bar\nabla$ and $D$, respectively.
For a smooth function $f$, we write $\nabla_g f$ and
$\bar\nabla_{\bar g}f$ for its gradients with respect to
$g$ and $\bar g$, respectively, and use the convention
$$
 \Delta_{\bar g}f
 =\diver_{\bar g}\bigl(\bar\nabla_{\bar g}f\bigr).
$$
Choose the downward unit normal
$$
 N=\frac{(\nabla_g u_\tau,-1)}{W}.
$$
We define the second fundamental form $A$ and the mean curvature
$H$ of $\Gamma_\tau$ with respect to the chosen unit normal $N$ by
$$
 A(Y_1,Y_2)=\langle D_{Y_1}N,Y_2\rangle_G,
 \qquad
 H=\tr_{\bar g}A,
$$
where $Y_1$ and $Y_2$ are tangent vectors to $\Gamma_\tau$.
Under the identification of $\Gamma_\tau$ with $X$ via the graph
parametrization, $A$ and $H$ are given by
$$
 A=\frac{\Hess_g u_\tau}{W},
 \qquad
 H=\diver_g\left(\frac{\nabla_g u_\tau}{W}\right).
$$

Since $t=u_\tau$ on $\Gamma_\tau$, the regularized equation
\eqref{eq:jg-dirichlet} becomes
\begin{equation}\label{eq:jg-H-height}
 H=\tau t-(n-1)-w^2\qquad\text{ on } \Gamma_\tau.
\end{equation}
By \eqref{eq:jg-height-global} and $0<w\le1$, we have
\begin{equation}\label{eq:jg-H-bound}
 -n\le H\le1-w^2,
 \qquad
 |H|\le n.
\end{equation}

The tangential projection of the vertical unit vector field
$\partial_t$ is
\begin{equation}\label{eq:jg-T}
 T=(\partial_t)^\top
   =\partial_t+wN
   =\bar\nabla_{\bar g}t.
\end{equation}
Note $|T|_{\bar g}^2=1-w^2.$
For vector fields $Y_1$ and $Y_2$ tangent to $\Gamma_\tau$, with our convention for $A$, the Gauss formula is
$$
 D_{Y_1}Y_2
 =\bar\nabla_{Y_1}Y_2-A(Y_1,Y_2)N.
$$
Since $\partial_t$ is parallel in $X\times\R$, differentiating
$\partial_t=T-wN$ along $Y_1$ and separating the tangential
and normal components gives
\begin{equation}\label{eq:tangent-normal}
 \bar\nabla_{Y_1}T=wD_{Y_1}N,
 \qquad
 Y_1(w)=-A(Y_1,T).
\end{equation}

Let $R^G$ denote the Riemann curvature tensor of
$(X\times\mathbb R,G)$, with the convention
$$
 R^G(U_1,U_2)U_3
 =
 D_{U_1}D_{U_2}U_3
 -D_{U_2}D_{U_1}U_3
 -D_{[U_1,U_2]}U_3
$$
for vector fields $U_1,\, U_2,\, U_3$ on $X\times\mathbb R$.
In a local $\bar g$-orthonormal frame $e_1,\ldots,e_n$,
the Codazzi equation is
$$
 (\bar\nabla_iA)_{jk}-(\bar\nabla_jA)_{ik}
 =
 \langle R^G(e_i,e_j)N,e_k\rangle_G.
$$
Contracting the first and third indices yields, for any
vector field $Y_1$ tangent to $\Gamma_\tau$,
\begin{equation}\label{eq:codazzicontra}
 \begin{aligned}
 (\diver_{\bar g}A)(Y_1)
 &=
 \sum_{i=1}^n(\bar\nabla_{e_i}A)(e_i,Y_1)\\
 &=
 \langle\bar\nabla_{\bar g}H,Y_1\rangle_{\bar g}
 +\Ric_G(Y_1,N),
 \end{aligned} 
\end{equation}
where $\Ric_G$ is the Ricci curvature
tensor of $(X\times\mathbb R,G)$.
Taking the divergence of $dw=-A(\,\cdot\,,T)$ and using
\eqref{eq:tangent-normal} and \eqref{eq:codazzicontra}, we obtain
\begin{equation}\label{eq:deltaw}
 \begin{aligned}
 \Delta_{\bar g}w
 &=
 -(\diver_{\bar g}A)(T)
 -\sum_{i=1}^n A(e_i,\bar\nabla_{e_i}T)\\
 &=
 -\langle\bar\nabla_{\bar g}H,T\rangle_{\bar g}
 -\Ric_G(T,N)-w|A|_{\bar g}^2.
 \end{aligned}
\end{equation}
Because $G$ is a product metric,
$\Ric_G(\partial_t,\cdot)=0$.
Writing
$$
 N_0=\frac{\nabla_g u_\tau}{W}
$$
for the horizontal component of $N$, we have
$$
 \Ric_G(T,N)
 =w\Ric_G(N,N)
 =w\Ric_g(N_0, N_0).
$$
Consequently, \eqref{eq:deltaw} becomes
\begin{equation}\label{eq:deltaw-new}
 \Delta_{\bar g}w
 =
 -\langle\bar\nabla_{\bar g}H,T\rangle_{\bar g}
 -\bigl(|A|_{\bar g}^2+\Ric_g(N_0, N_0)\bigr)w.
\end{equation}
On the other hand, taking the gradient of
\eqref{eq:jg-H-height} with respect to $\bar g$ gives
\begin{equation}\label{eq:nablaH}
 \bar\nabla_{\bar g}H
 =\tau T-2w\bar\nabla_{\bar g}w.
\end{equation}
Define the drift vector field $T_0$ and the operator $\mathcal P$ for $f\in C^2(\Gamma_\tau)$ by
\begin{equation}\label{eq:jg-P-def}
 T_0=2wT,
 \qquad
 \mathcal Pf=\Delta_{\bar g}f
    -\langle T_0,\bar\nabla_{\bar g}f\rangle_{\bar g}.
\end{equation}
Using $|T|_{\bar g}^2=1-w^2$, \eqref{eq:deltaw-new} and \eqref{eq:nablaH}, we have
$$
 \mathcal Pw
 =
 -\bigl(|A|_{\bar g}^2+\Ric_g(N_0, N_0)\bigr)w
 -\tau(1-w^2).
$$
Moreover,
\begin{equation}\label{eq:jg-drift-bound}
 |T_0|_{\bar g}^2=4w^2(1-w^2)\le1.
\end{equation}
Further, we obtain
\begin{equation}\label{eq:jg-PW-exact}
 \begin{aligned}
 \mathcal PW
 &=-w^{-2}\mathcal Pw
   +2w^{-3}|\bar\nabla_{\bar g}w|_{\bar g}^2\\
 &=\frac{2|\bar\nabla_{\bar g}W|_{\bar g}^2}{W}
   +\bigl(|A|_{\bar g}^2+\Ric_g(N_0, N_0)\bigr)W
   +\tau(W^2-1).
 \end{aligned}
\end{equation}

Since $X$ is compact, choose $\kappa\ge0$ such that
$\Ric_g\ge-\kappa g$ on $X$.
By
$$
 |N_0|_g^2=1-W^{-2}\le1
$$
we know $\Ric_g(N_0, N_0)\ge-\kappa$.
Dropping the nonnegative terms $W|A|_{\bar g}^2$ and
$\tau(W^2-1)$ from \eqref{eq:jg-PW-exact}, we conclude that
\begin{equation}\label{eq:jg-PW-inequality}
 \mathcal PW\ge
 \frac{2|\bar\nabla_{\bar g}W|_{\bar g}^2}{W}
 -\kappa W.
\end{equation}

We record the following consequence, which will be used twice below.
Let $\eta$ be a smooth function defined in a
neighborhood of an interior point $p$ of $\Gamma_\tau$.
Suppose that $\eta(p)>0$ and that $\eta W$ attains a local
maximum at $p$.
By \eqref{eq:jg-PW-inequality} we obtain, with all quantities evaluated at $p$,
\begin{equation}\label{eq:jg-max-inequality}
 \begin{aligned}
 0\ge\mathcal P(\eta W)
 &=\eta\,\mathcal P W+W\,\mathcal P\eta
   +2\langle
      \bar\nabla_{\bar g}\eta,
      \bar\nabla_{\bar g}W
     \rangle_{\bar g}\\
 &\ge W(\mathcal P\eta-\kappa\eta)
   +\frac{2\eta}{W}|\bar\nabla_{\bar g}W|_{\bar g}^2
   +2\langle
      \bar\nabla_{\bar g}\eta,
      \bar\nabla_{\bar g}W
     \rangle_{\bar g}\\
 &=W(\mathcal P\eta-\kappa\eta).
 \end{aligned}
\end{equation}
The last equality follows from the gradient relation at $p$,
which makes the two gradient terms cancel.
Since $W(p)>0$, we conclude that
\begin{equation}\label{eq:jg-max-inequality-1}
(\mathcal P\eta)(p)\le\kappa\eta(p).
\end{equation}

\medskip
\noindent\textbf{Step 4: The interior gradient estimate.}
We first estimate $\mathcal P t$ and $\mathcal P b$, where
$t$ is the vertical coordinate restricted to $\Gamma_\tau$
and $b$ is a smooth function on the base $X$, extended
independently of $t$ and then restricted to $\Gamma_\tau$.
For a smooth function $f$ on $X\times \R$, also denoting its restriction
to $\Gamma_\tau$ by $f$, the Gauss formula gives
$$
 \Hess_{\bar g}f(Y_1,Y_2)
 =
 \Hess_Gf(Y_1,Y_2)-A(Y_1,Y_2)N(f)
$$
for vector fields $Y_1,Y_2$ tangent to $\Gamma_\tau$.
Thus, in a local $\bar g$-orthonormal frame $e_1,\ldots,e_n$,
\begin{equation}\label{eq:restdelta}
 \Delta_{\bar g}f
 =
 \sum_{i=1}^n\Hess_Gf(e_i,e_i)-HN(f).
\end{equation}
Since $\Hess_Gt=0$ and $N(t)=-w$, and the restriction of $t$
to $\Gamma_\tau$ agrees with $u_\tau$ under the graph
parametrization, it follows that
\begin{equation}\label{eq:jg-delta-t}
 \Hess_{\bar g}u_\tau=wA,
 \qquad
 \Delta_{\bar g}u_\tau=wH.
\end{equation}
Using \eqref{eq:jg-H-bound}, \eqref{eq:jg-T} and
\eqref{eq:jg-drift-bound}, we obtain
\begin{equation}\label{eq:jg-Pt}
 \mathcal P u_\tau=\frac HW-2w(1-w^2),
 \qquad
 |\mathcal P u_\tau|\le n+1,
 \qquad
 \mathcal P u_\tau\ge-\frac nW-1.
\end{equation}

On the other hand, the inverse graph metric gives
\begin{equation}\label{eq:restrict}
 \langle T,\bar\nabla_{\bar g}b\rangle_{\bar g}
 =
 \frac{\langle du_\tau,db\rangle_g}{W^2}.
\end{equation}
If
$$
 |db|_g\le B_1,
 \qquad
 \|\Hess_gb\|_{\mathrm{op},g}\le B_2
$$
for nonnegative constants $B_1$ and $B_2$, then \eqref{eq:restdelta}, \eqref{eq:restrict}, and the bounds $|H|\le n$ and $|T_0|_{\bar g}\le1$, implies
\begin{equation}\label{eq:jg-Pb}
 \bigl|\langle T,\bar\nabla_{\bar g}b\rangle_{\bar g}\bigr|
 \le\frac{B_1}{W},
 \qquad
 |\mathcal P b|\le nB_2+(n+1)B_1.
\end{equation}
Here we used
$|\bar\nabla_{\bar g}b|_{\bar g}\le|db|_g$,
$|N(b)|\le|db|_g$, and
$$
 \left|\sum_{i=1}^n\Hess_Gb(e_i,e_i)\right|\le nB_2.
$$

Fix a compact set $P_1\subset X^\circ$ and a compact
neighborhood $P_2\subset X^\circ$ of $P_1$.
By compactness, we may choose $\rho>0$ such that the
$\rho$-neighborhood of $P_1$ is contained in $P_2$, and
for every $x_0\in P_2$, the closed ball $\overline{B_\rho(x_0)}$
is contained in $X^\circ$ and lies in a normal coordinate
neighborhood centered at $x_0$.
After decreasing $\rho$ if necessary, we may further assume
that the squared distance functions $d_g(x_0,\cdot)^2$
have bounded $g$-Hessians on $B_\rho(x_0)$, with a bound
independent of the center $x_0\in P_2$.

We claim that, for every $L\ge0$, there is a constant $M_L$
independent of $\tau$ such that
\begin{equation}\label{eq:jg-point-slope}
 W(x_0)\le M_L
 \quad\text{whenever }x_0\in P_2
 \text{ and }u_\tau(x_0)\le L.
\end{equation}
Fix such a point $x_0$ and write $d(x)=d_g(x_0,x)$ on
$B_\rho(x_0)$. Set
$$
 \alpha=\frac{L+\rho}{\rho^2},
 \qquad
 b=-\alpha d^2,
 \qquad
 \phi=L+\rho-u_\tau+b.
$$
The nonnegativity of $u_\tau$ and the height bound at $x_0$
give
\begin{equation}\label{eq:evaphi}
 \phi(x_0)\ge\rho,
 \qquad
 \phi\le L+\rho,
 \qquad
 \phi|_{\partial B_\rho(x_0)}=-u_\tau\le0.
\end{equation}
In \eqref{eq:jg-Pb}, we may take $B_1=2\alpha\rho$
and choose $B_2$ independently of the center $x_0\in P_2$.
Set
$$
 C_1=n+1+nB_2+(n+1)B_1.
$$
By \eqref{eq:jg-Pt} and \eqref{eq:jg-Pb}, we have
$$
 \mathcal P\phi=-\mathcal P u_\tau+\mathcal P b\ge-C_1.
$$
Moreover, since $\bar\nabla_{\bar g}u_\tau=T$,
\begin{equation}\label{eq:phiineq}
 |\bar\nabla_{\bar g}\phi|_{\bar g}^2
 =|-T+\bar\nabla_{\bar g}b|_{\bar g}^2
 \ge1-W^{-2}-\frac{2B_1}{W}.
\end{equation}
Choose
$$
 B_*=2+8B_1,
 \qquad
 \kappa_*=4C_1+2\sqrt{\kappa+1}.
$$
Whenever $W\ge B_*$, \eqref{eq:phiineq} implies
$|\bar\nabla_{\bar g}\phi|_{\bar g}^2\ge1/2$.

Now let
$$
 \eta=\max\{e^{\kappa_*\phi}-1,0\}.
$$
The continuous function $\eta W$ vanishes on
$\partial B_\rho(x_0)$ and is positive at $x_0$.
It therefore attains a positive maximum at an interior
point $q$, where $\eta$ is smooth and positive.
If $W(q)\ge B_*$, then at $q$ we have
$$
 \begin{aligned}
 \mathcal P\eta-\kappa\eta
 &=e^{\kappa_*\phi}
   \bigl(
      \kappa_*\mathcal P\phi
      +\kappa_*^2|\bar\nabla_{\bar g}\phi|_{\bar g}^2
      -\kappa
   \bigr)+\kappa\\
 &\ge
 e^{\kappa_*\phi}
 \left(\frac{\kappa_*^2}{2}-\kappa_*C_1-\kappa\right)+\kappa
 >0.
 \end{aligned}
$$
The last inequality comes from $\kappa_*^2/2-\kappa_*C_1-\kappa\ge1$.
This contradicts \eqref{eq:jg-max-inequality-1}.
Hence $W(q)<B_*$.
Using the maximality of $\eta W$ at $q$,
$\phi(x_0)\ge\rho$ and \eqref{eq:evaphi}, we conclude that
\begin{equation}\label{eq:jg-interior-estimate}
 W(x_0)
 \le
 B_*\frac{e^{\kappa_*(L+\rho)}-1}{e^{\kappa_*\rho}-1}
 =:M_L.
\end{equation}
The constant $M_L$ depends only on $L$, $P_1$, $P_2$,
and the fixed metric $g$, and is independent of $\tau$. This proves \eqref{eq:jg-point-slope}.

\medskip
\noindent\textbf{Step 5: Uniform regularity estimates at bounded height.}
We first turn the pointwise estimate \eqref{eq:jg-interior-estimate} into uniform
estimates on a neighborhood of fixed radius.
Let $P_1$, $P_2$ and $\rho$ be as in Step~4.
Suppose that $x_0\in P_1$ and $u_\tau(x_0)\le L$, and set
$$
 \varepsilon_L
 =
 \min\left\{\frac{\rho}{2},\frac{1}{2M_{L+1}}\right\}.
$$
We claim that $u_\tau\le L+1$ on $B_{\varepsilon_L}(x_0)$.
Otherwise, choose $y\in B_{\varepsilon_L}(x_0)$ such that
$u_\tau(y)>L+1$, and let $\gamma$ be a minimizing geodesic
from $x_0$ to $y$, parametrized by arc length.
Since $u_\tau(x_0)\le L$, there exists $0<\ell<\varepsilon_L$ such that
$u_\tau(\gamma(\ell))=L+1,$ and $ u_\tau(\gamma(s))<L+1$ for $0\le s<\ell.$
This segment lies in $P_2$, so \eqref{eq:jg-point-slope}
gives $|du_\tau|_g\le M_{L+1}$ along the segment $\gamma([0,\ell])$.
Consequently,
$$
 1
 \le (L+1)-u_\tau(x_0)
 =\int_0^\ell du_\tau(\dot\gamma(s))\,ds
 \le M_{L+1}\ell
 \le\frac12,
$$
which is a contradiction. Therefore, by \eqref{eq:jg-point-slope} again, we conclude that
\begin{equation}\label{eq:jg-neighborhood}
 0\le u_\tau\le L+1,
 \qquad
 |du_\tau|_g\le M_{L+1}
 \quad\text{on }B_{\varepsilon_L}(x_0).
\end{equation}

To obtain estimates up to $S$, consider the fixed collar
$$
 \mathcal C=\{0\le s\le\frac \delta 2\}\subset X.
$$
By \eqref{eq:jg-collar-height}, we have
$0\le u_\tau\le L_c$ on $\mathcal C$.
The boundary gradient estimate \eqref{eq:jg-boundary-slope}
bounds $W$ uniformly on $S$.
Since $\{s=\delta/2\}$ is a fixed compact hypersurface
in $X^\circ$, the interior gradient estimate \eqref{eq:jg-point-slope}, applied with height bound $L_c$, also bounds $W$ uniformly on this hypersurface.
Thus there exists a constant $W_*$, independent of $\tau$,
such that
$$
 W\le W_*
 \qquad\text{on }\partial\mathcal C
 =S\cup\{s=\frac\delta 2\}.
$$
Choose $\lambda>0$ such that
\begin{equation}\label{eq:lambda}
 \frac34\lambda^2-\left(\frac n2+1\right)\lambda-\kappa>0,
\end{equation}
and consider $e^{\lambda u_\tau}W$ on the graph over
$\mathcal C$.
At an interior maximum point of $e^{\lambda u_\tau}W$,
we apply \eqref{eq:jg-max-inequality} with
$\eta=e^{\lambda u_\tau}$ to obtain
$$
 \begin{aligned}
 0
 \ge \frac{\mathcal P\eta-\kappa\eta}{\eta}&=\lambda\mathcal P u_\tau
   +\lambda^2|T|_{\bar g}^2-\kappa\\
 &\ge \lambda^2(1-W^{-2})
   -\frac{\lambda n}{W}-\lambda-\kappa.
 \end{aligned}
$$
Here the equality follows from $\bar\nabla_{\bar g}u_\tau=T$, while the last inequality
uses \eqref{eq:jg-Pt} and $|T|_{\bar g}^2=1-W^{-2}$.
If $W\ge2$ at this point, the last expression is positive by \eqref{eq:lambda}, contradicting the left hand side inequality. Thus $W<2$ at any interior maximum point of
$e^{\lambda u_\tau}W$, while $W\le W_*$ at any boundary
maximum point.
Since $0\le u_\tau\le L_c$ on $\mathcal C$, the maximum
value of $e^{\lambda u_\tau}W$ is at most
$e^{\lambda L_c}\max\{2,W_*\}$. Therefore,
\begin{equation}\label{eq:jg-collar-gradient}
 W\le e^{\lambda u_\tau}W
 \le e^{\lambda L_c}\max\{2,W_*\}
 \quad\text{on }\mathcal C.
\end{equation}

The estimates \eqref{eq:jg-neighborhood} on
$B_{\varepsilon_L}(x_0)$, together with
\eqref{eq:jg-collar-height} and \eqref{eq:jg-collar-gradient}
on $\mathcal C$, give height and gradient bounds independent
of $\tau$.
In particular, these gradient bounds ensure that
\eqref{eq:jg-dirichlet} is uniformly elliptic on both sets,
with ellipticity constants independent of $\tau$.

We apply standard interior and boundary $C^{1,\alpha}$ estimates \cite[Theorem 13.6, and estimate (13.41) following Theorem 13.7]{GT} to \eqref{eq:jg-dirichlet}, using the height and gradient
bounds \eqref{eq:jg-neighborhood},
\eqref{eq:jg-collar-height}, \eqref{eq:jg-collar-gradient}, and the boundary condition $u_\tau=0$ on $S$. The resulting H\"older bounds for the coefficients allow us to apply the interior and local boundary Schauder estimates \cite[Theorem 6.2 and Corollary 6.7]{GT}.
Then, a bootstrap argument yields, for every integer $k\ge2$,
$$
 \|u_\tau\|_{C^k(\overline{B_{\varepsilon_L/2}(x_0)})}
 \le C_{k,L},
 \qquad
 \|u_\tau\|_{C^k(\{0\le s\le\delta/4\})}
 \le C_{k,S},
$$
Here all $C^k$ norms are taken with respect to $g$.
The first estimate holds whenever $x_0\in P_1$ and
$u_\tau(x_0)\le L$.
The constant $C_{k,L}$ depends only on $k$, $L$, $P_1$,
$P_2$, and $g$, while $C_{k,S}$ depends only on $k$, $g$,
$S$, and the fixed constants $\delta$, $A_0$ from Step 2.
Both constants are independent of $\tau$.

\medskip
\noindent\textbf{Step 6: Passing to the limit.} By monotonicity \eqref{eq:jg-monotone}, the pointwise limit
\begin{equation}\label{eq:jg-limit-def}
 u(x)
 =
 \lim_{\tau\to 0+}u_\tau(x)
 =
 \sup_{0<\tau<1}u_\tau(x)
 \in[0,+\infty]
\end{equation}
exists for every $x\in X$. Define
$$
 \Omega=\{x\in X:u(x)<+\infty\}.
$$
Fix $x_0\in\Omega\cap X^\circ$, and choose compact sets
$P_1$, $P_2$ as in Step 4 with $x_0\in P_1$.
Choose $L>u(x_0)$. Since
$u_\tau(x_0)\le u(x_0)<L$ for every $\tau$,
\eqref{eq:jg-neighborhood} gives
$u_\tau\le L+1$ on $B_{\varepsilon_L}(x_0)$,
where $\varepsilon_L$ is independent of $\tau$.
Taking the pointwise limit, we obtain $u\le L+1$ on $B_{\varepsilon_L}(x_0)$.
So $B_{\varepsilon_L}(x_0)\subset\Omega$.
Moreover, by \eqref{eq:jg-collar-bounds} we have
$$
 u(x)\le U(s(x))<+\infty
 \qquad\text{whenever }0\le s(x)<\delta.
$$
Thus $\Omega$ contains the collar $\{0\le s<\delta\}$
and is relatively open in $X$.

The interior and boundary regularity estimates of Step 5,
together with a finite covering argument, give uniform
bounds of every derivative order on each compact subset
of $\Omega$, including compact subsets that meet $S$.
By the Arzel\`a--Ascoli theorem and a diagonal argument,
every sequence $\{u_{\tau_j}\}$ with $\tau_j\to0$ admits
a subsequence converging smoothly on compact subsets
of $\Omega$. Every such limit agrees with $u$ by \eqref{eq:jg-limit-def}.
Thus $u$ is smooth in $\Omega\cap X^\circ$ and extends
smoothly to the boundary $S$. Since every sequence $\{u_{\tau_j}\}$ with $\tau_j\to0$
has a subsequence converging smoothly to $u$, the entire
family $\{u_\tau\}$ converges to $u$:
\begin{equation}\label{eq:jg-smooth-limit}
 u_\tau\longrightarrow u
 \quad\text{in }C^k(K)
 \quad\text{as }\tau\to 0+,
\end{equation}
for every compact subset $K\subset\Omega$ and every
integer $k\ge0$.

The $C^0$ convergence in \eqref{eq:jg-smooth-limit}
implies that $u_\tau$ is uniformly bounded on each compact
subset of $\Omega$, and hence $\tau u_\tau\to0$ uniformly
there. Passing to the limit in \eqref{eq:jg-dirichlet}
using the $C^2$ convergence in \eqref{eq:jg-smooth-limit},
we obtain
$$
 \diver_g\left(
 \frac{\nabla_g u}{\sqrt{1+|du|_g^2}}
 \right)
 +(n-1)+\frac{1}{1+|du|_g^2}=0
 \qquad\text{in }\Omega\cap X^\circ,
$$
which is equivalent to \eqref{eq:Jang}.
The nonnegativity in \eqref{eq:jg-height-global} and the
boundary condition in \eqref{eq:jg-dirichlet} also pass
to the limit, giving $u\ge0$ on $\Omega$ and $u=0$ on $S$. In fact, $u>0$ on $\Omega\cap X^\circ$.
Indeed, if $u(p)=0$ at some interior point $p$,
then $u$ attains a minimum at $p$, so
$du(p)=0$ and $\Delta_g u(p)\ge0$.
At this point, however, the limiting equation reduces to
$$
 \Delta_g u(p)+n=0,
$$
which leads to a contradiction.

\medskip
\noindent\textbf{Step 7: Properness of the limiting graph.}
We first show that $\Omega$ is connected.
By \eqref{eq:jg-limit-def}, for every finite $L\ge0$,
\begin{equation}\label{eq:jg-closed-sublevels}
 E_L
 :=
 \{x\in\Omega:u(x)\le L\}
 =
 \bigcap_{0<\tau<1}
 \{x\in X:u_\tau(x)\le L\}.
\end{equation}
Each $u_\tau$ is continuous on $X$, so the right hand
side is closed in the compact manifold $X$.
Hence $E_L$ is compact and contained in $\Omega$.

Since $S$ is connected and contained in $\Omega$,
it lies in a single connected component of $\Omega$.
If $\Omega$ is disconnected, there exists a connected component $\Omega'$ disjoint from $S$.
As a connected component of $\Omega$,
$\Omega'$ is open and closed in $\Omega$.
Choose $x_0\in \Omega'$ and set $L=u(x_0)$.
Then $E_L\cap \Omega'$ is nonempty and compact, so $u$ attains
its minimum on this set at some point $p$.
Since $u>L$ on $\Omega'\setminus E_L$, this is also a minimum
of $u$ on $\Omega'$.
Moreover, $\Omega'$ is open in $X^\circ$, so $p$ is an
interior minimum and
$$
 du(p)=0,\qquad \Delta_g u(p)\ge0.
$$
However, the limiting equation at $p$ gives
$\Delta_g u(p)=-n<0$, a contradiction.
Thus $\Omega$ is connected.

Define the limiting graph by
$$
 \Gamma
 :=
 \operatorname{graph}(u)
 =
 \{(x,u(x)):x\in\Omega\}.
$$
We equip $\Gamma$ with the metric induced by $G=g+dt^2$. The graph parametrization $x\mapsto(x,u(x))$ is a smooth
embedding, with inverse on $\Gamma$ given
by the projection onto $X$. Since $\Omega$ is relatively open in $X$ and contains
$S=\partial X$, its manifold boundary is $S$.
Together with $u=0$ on $S$, this shows that the manifold
boundary of $\Gamma$ is $S\times\{0\}$.

For every finite $L\ge0$, the set
$$
 \Gamma\cap\bigl(X\times[0,L]\bigr)
 =
 \{(x,u(x)):x\in E_L\}
$$
is the image of the compact set $E_L$ under the graph
parametrization and is therefore compact.
Now let $K\subset X\times\mathbb R$ be any compact set,
and choose a finite constant $L\ge0$ that bounds its
$t$-coordinate from above.
Since $u\ge0$ and $K$ is closed in $X\times\mathbb R$,
the intersection $\Gamma\cap K$ is a closed
subset of the compact set
$\Gamma\cap(X\times[0,L])$.
Thus $\Gamma\cap K$ is compact, which implies
$\Gamma$ is properly embedded in
$X\times\mathbb R$.

Finally, suppose that $x_j\in\Omega$ converges in $X$
to a point $x_\infty\notin\Omega$.
If $u(x_j)$ does not tend to $+\infty$, then there exist
a finite constant $L\ge0$ and a subsequence
$\{x_{j_\ell}\}$ contained in $E_L$.
Since $E_L$ is compact and closed in $X$,
we have $x_\infty\in E_L\subset\Omega$, which is a contradiction. Therefore $u(x_j)\to+\infty$.

We conclude that $\Omega$ is connected and that
$\Gamma$ is properly embedded in
$X\times\mathbb R$ with manifold boundary $S\times\{0\}$.
Moreover, $u$ tends to $+\infty$ as $x\in\Omega$
approaches any point of
$\overline{\Omega}^{X}\setminus\Omega$.

\medskip
\noindent\textbf{Step 8: Completeness of the limiting graph.}
For $p,\, q\in\Gamma$, let $d_\Gamma(p,q)$ be the infimum
of the $\bar g$-lengths of all piecewise smooth curves in
$\Gamma$ joining $p$ to $q$.
The points $p,\, q$ may lie on the boundary $S\times\{0\}$,
and the curves are allowed to meet or run along this boundary.
For every piecewise smooth curve $\gamma$ in $\Gamma$,
the vertical coordinate $t$ satisfies
$$
 |(t\circ\gamma)'|
 =
 |\langle\partial_t,\dot\gamma\rangle_G|
 \le |\dot\gamma|_G
 =
 |\dot\gamma|_{\bar g}.
$$
Integrating along such curves and taking the infimum
over all curves joining $p$ and $q$, we obtain
\begin{equation}\label{eq:jg-height-lip}
 |t(p)-t(q)|\le d_\Gamma(p,q).
\end{equation}

Let $\{p_j\}$ be a $d_\Gamma$-Cauchy sequence in $\Gamma$.
By \eqref{eq:jg-height-lip}, the sequence $\{t(p_j)\}$
is a Cauchy sequence in $\mathbb R$ and therefore bounded.
Since $u\ge0$, the sequence $\{p_j\}$ lies in
$\Gamma\cap(X\times[0,L])$ for some finite constant $L\ge0$.
This set is compact by Step 7.
Hence, there is a subsequence $\{p_{j_\ell}\}$ that
converges in the manifold topology to some $p\in\Gamma$.

Choose local coordinates $y$ near $p$, with $y(p)=0$,
whose image is a Euclidean ball if $p$ is an interior
point, or a half ball if $p\in S\times\{0\}$.
Then
$$
 |y(p_{j_\ell})|
 =
 |y(p_{j_\ell})-y(p)|
 \longrightarrow0\qquad\text{ as }\ell\to\infty.
$$
After shrinking the coordinate neighborhood,
the smoothness of the induced metric gives a
constant $C>0$ such that
$$
 \bar g\le C\sum_i(dy^i)^2.
$$
For all sufficiently large $\ell$, the straight segment
from $0$ to $y(p_{j_\ell})$ lies in the coordinate image.
Its inverse image is a curve in $\Gamma$ joining $p$
to $p_{j_\ell}$, with length at most
$\sqrt{C}\,|y(p_{j_\ell})|$.
Consequently,
$$
 d_\Gamma(p_{j_\ell},p)
 \le \sqrt{C}\,|y(p_{j_\ell})|
 \longrightarrow0.
$$
Moreover, since $\{p_j\}$ is a $d_\Gamma$-Cauchy sequence,
for every $\varepsilon>0$ there exists $N_*$ such that
$$
 d_\Gamma(p_j,p_k)<\frac{\varepsilon}{2}
 \qquad\text{for all }j,\, k\ge N_*.
$$
Choose $\ell$ sufficiently large that $j_\ell\ge N_*$ and
$d_\Gamma(p_{j_\ell},p)<\varepsilon/2$.
Then, for every $j\ge N_*$,
$$
 d_\Gamma(p_j,p)
 \le d_\Gamma(p_j,p_{j_\ell})
    +d_\Gamma(p_{j_\ell},p)
 <\varepsilon.
$$
Thus the entire sequence $\{p_j\}$ converges to $p$ in $d_\Gamma$.
We conclude that the limiting graph $\Gamma$,
including its boundary $S\times\{0\}$, is complete
with respect to its induced length metric.
\end{proof}

\begin{remark}
The properness and completeness established in
Lemma \ref{lem:jang} concern the smooth limiting graph
$\Gamma$, regarded as a manifold with boundary
$S\times\{0\}$. We also establish that $u$ is smooth
in $\Omega\cap X^\circ$ and up to $S$, so that
$\Gamma$ is a smooth manifold with boundary
$S\times\{0\}$. No regularity is asserted for the set
$\overline{\Omega}^{X}\setminus\Omega$.
\end{remark}

For the remainder of this section, let $u$, $\Omega$ and $\Gamma$ be as in Lemma \ref{lem:jang}. We identify $\Gamma$ with $\Omega$ via the graph
parametrization $\mathcal G: \Omega\to\Gamma$, where $\mathcal G$ is defined by $\mathcal G(x)=(x,u(x))$. We note the inverse of $\mathcal G$ is the projection $\pi:\Gamma\to\Omega$.
Under this identification, each tangent vector
$V\in T_x\Omega$ corresponds to
$d\mathcal G_x(V)=V+du_x(V)\partial_t\in T_{\mathcal G(x)}\Gamma$,
and covariant tensors on $\Omega$ are regarded as
tensors on $\Gamma$ via the pullback $\pi^*$.
We suppress these pushforward and pullback maps
in the notation.

Let $G=g+dt^2$ be the product metric on
$X\times\mathbb R$, and let $D$ denote its
Levi--Civita connection.
Set
$$
 W=\sqrt{1+|du|_g^2},
 \qquad
 \bar g=g+du^2.
$$
Then $\bar g$ is the induced metric on $\Gamma$. Let $\bar\nabla$ and $\nabla$ denote the Levi--Civita connections of $\bar g$ and $g$, respectively.
Choose the downward unit normal
$$
 N_\Gamma=\frac{(\nabla_g u,-1)}{W}.
$$
With the convention
$$
 A_\Gamma(Y_1, Y_2)
 :=
 \langle D_{Y_1}N_\Gamma, Y_2\rangle_G,
 \qquad Y_1,\, Y_2\in T\Gamma,
$$
the second fundamental form of $\Gamma$ is
$$
 A_\Gamma=\frac{\Hess_g u}{W}.
$$
In this notation, the Jang equation \eqref{eq:Jang},
corresponding to the initial data tensor $-g$,
takes the form
$$
 \tr_{\bar g}(A_\Gamma+g)=0.
$$
Define the one-form $q$ on $\Gamma$ by
$$
 q=(A_\Gamma+g)
 \left(\frac{\nabla_g u}{W},\,\cdot\right).
$$
In local coordinates on $\Omega$, we have
\begin{equation}\label{eq:jang-tensors}
 q_i
 =
 \frac{u^j}{W}
 \bigl((A_\Gamma)_{ij}+g_{ij}\bigr),
\end{equation}
where $u_j=\partial_j u$ and $u^j=g^{jk}u_k$.

Specializing the local Schoen--Yau scalar curvature calculation \cite[Section 2]{SchoenYau} to the initial data tensor $-g$ gives
\begin{equation}\label{eq:SY}
 \Sc_{\bar g}
 =
 \Sc_g+n(n-1)
 +|A_\Gamma+g|_{\bar g}^2
 +2|q|_{\bar g}^2
 -2\diver_{\bar g}q,
\end{equation}
here $\Sc_{\bar g}$ denotes the scalar curvature of the Jang graph with its induced metric $\bar g=g+du^2$.

Since $\partial\Gamma=S\times\{0\}$, we identify $\partial\Gamma$ with
$S$. Let $\eta_\Gamma$ denote the outward unit conormal of
$\partial\Gamma$ in $(\Gamma,\bar g)$, and let
$H_{\partial\Gamma}$ denote the mean curvature of
$\partial\Gamma\subset(\Gamma,\bar g)$ with respect to
$\eta_\Gamma$. Unless otherwise indicated, norms and inner products on $\Gamma$ are taken with respect to $\bar g$. Differential operators refer to the connection or
metric specified in their notation.

For $\xi\in C^\infty(S)$, define the quadratic form
$\mathcal E_\xi$ on $C_c^\infty(\Gamma)$ by
$$
 \mathcal E_\xi(v)
 =
 \int_\Gamma
 \left(
 |dv|_{\bar g}^2+\frac12\Sc_{\bar g}v^2
 \right)
 +\int_S(H_{\partial\Gamma}-\xi)v^2.
$$

We next transfer the scalar and mean curvature assumptions on $(X,g)$
to the complete Jang graph $\Gamma$. The Schoen--Yau identity
\eqref{eq:SY} allows the scalar curvature term in $\mathcal E_\xi$ to
be rewritten as a square together with nonnegative interior terms,
while the boundary geometry of the graph converts the lower bound for
$H_S$ into a sharp lower bound for
$H_{\partial\Gamma}-q(\eta_\Gamma)$.
\begin{lemma}\label{lem:graph-form}
Let $(X, g)$, $S$ and $\Gamma$ be as in Lemma \ref{lem:jang}.  Let $R>0$, and suppose
$$
 \Sc_g\ge-n(n-1)
 \qquad\text{and}\qquad
 H_S\ge(n-1)\coth R.
$$
Then, for every $\xi\in C^\infty(S)$ and
$v\in C_c^\infty(\Gamma)$,
\begin{equation}\label{eq:graph-square}
 \mathcal E_\xi(v)
 =
 \int_\Gamma|dv+vq|_{\bar g}^2
 +\frac12\int_\Gamma
 \left(
 \Sc_g+n(n-1)+|A_\Gamma+g|_{\bar g}^2
 \right)v^2
 +\int_S
 \left(
 H_{\partial\Gamma}-q(\eta_\Gamma)-\xi
 \right)v^2 .
\end{equation}
Moreover,
$$
 H_{\partial\Gamma}-q(\eta_\Gamma)
 \ge\frac{n-1}{\sinh R}
 \qquad\text{on }S.
$$
\end{lemma}
\begin{proof}
We first prove the boundary estimate. Since $u=0$ on $S$, its tangential derivatives vanish there.  Let $\nu$ be the outward unit normal of $S=\partial X$. Using $\langle\nabla_g u,\nu\rangle_g=\nu u$, we have
$$
 \nabla_g u=(\nu u)\nu,
 \qquad
 W=\sqrt{1+(\nu u)^2}
 \qquad\text{on }S.
$$
Under the graph parametrization, the outward unit conormal of
$\partial\Gamma$ in $\Gamma$ is
$$
 \eta_\Gamma
 =
 \frac{\nu+(\nu u)\partial_t}{W}.
$$
Since the induced metric on $\partial\Gamma$ agrees with $g_S$, we have
$$
 H_{\partial\Gamma}=\frac{H_S}{W},
 \qquad
 \tr_{g_S}(A_\Gamma|_{TS\times TS})=\frac{(\nu u)H_S}{W}.
$$

Decomposing the $\bar g$-trace of the Jang equation \eqref{eq:Jang} into directions tangent to $S$ and the outward conormal direction gives
$$
 \frac{A_\Gamma(\nu,\nu)+1}{W^2}
 +(n-1)+\frac{(\nu u)H_S}{W}=0.
$$
Therefore
$$
 \frac{A_\Gamma(\nu,\nu)+1}{W^2}
 =
 -(n-1)-\frac{(\nu u)H_S}{W}.
$$
By the definition of $q$,
$$
\begin{aligned}
 q(\eta_\Gamma)
 &=
 \frac{\nu u}{W^2}
 \bigl(A_\Gamma(\nu,\nu)+1\bigr)\\
 &=
 -(n-1)\nu u
 -\frac{(\nu u)^2H_S}{W}.
\end{aligned}
$$
Consequently,
\begin{equation}\label{eq:graph-boundary}
\begin{aligned}
 H_{\partial\Gamma}-q(\eta_\Gamma)
 &=
 \frac{H_S}{W}
 +(n-1)\nu u
 +\frac{(\nu u)^2H_S}{W}\\
 &=
 H_S\sqrt{1+(\nu u)^2}+(n-1)\nu u.
\end{aligned}
\end{equation}
Since
$$
 H_S\ge(n-1)\coth R>n-1,
$$
minimizing the right hand side with respect to $\nu u\in\mathbb R$
gives
$$
 H_{\partial\Gamma}-q(\eta_\Gamma)
 \ge
 \sqrt{H_S^2-(n-1)^2}
 \ge
 \frac{n-1}{\sinh R}.
$$

We next prove \eqref{eq:graph-square}. Substituting the Schoen--Yau
identity \eqref{eq:SY} into the definition of $\mathcal E_\xi$ gives
$$
 \mathcal E_\xi(v)
 =
 \int_\Gamma
 \left(
 |dv|_{\bar g}^2
 +|q|_{\bar g}^2v^2
 -(\diver_{\bar g}q)v^2
 \right)+
 \frac12\int_\Gamma
 \left(
 \Sc_g+n(n-1)+|A_\Gamma+g|_{\bar g}^2
 \right)v^2
 +\int_S(H_{\partial\Gamma}-\xi)v^2.
$$
Since $v$ has compact support in $\Gamma$, integration by parts yields
$$
 -\int_\Gamma(\diver_{\bar g}q)v^2
 =
 2\int_\Gamma v\langle q,dv\rangle_{\bar g}
 -\int_Sq(\eta_\Gamma)v^2,
$$
with no contribution from the noncompact ends of $\Gamma$. Hence
$$
 |dv|_{\bar g}^2
 +2v\langle q,dv\rangle_{\bar g}
 +|q|_{\bar g}^2v^2
 =
 |dv+vq|_{\bar g}^2,
$$
which proves \eqref{eq:graph-square}.
\end{proof}

We will apply Corollary \ref{cor:positive-form} to the complete Jang graph $\Gamma$ to obtain a contradiction. It therefore remains to strengthen \eqref{eq:graph-square} to an estimate with a strictly positive interior weight and a strictly positive boundary term. The next lemma provides this strengthening.

\begin{lemma}\label{lem:anchor}
Let $n\ge3$, and let $(M^n,g_M)$ be a smooth, connected Riemannian manifold  with nonempty smooth compact boundary $S$. Let $\omega$ be a smooth
one-form, $V\ge0$ a smooth function on $M$, and $B\ge0$ a smooth
function on $S$. Define
$$
 \mathcal A(v)
 =
 \int_M|dv+v\omega|_{g_M}^2
 +\int_MVv^2
 +\int_SBv^2,
 \qquad \text{ for } v\in C_c^\infty(M).
$$
If $V$ or $B$ is positive somewhere, then there are a smooth function
$\varpi>0$ on $M$ and a constant $\varepsilon_0>0$ such that
\begin{equation}\label{eq:anchored-form}
 \mathcal A(v)\ge\int_M\varpi v^2+\varepsilon_0\int_Sv^2.
\end{equation}
\end{lemma}

\begin{proof}
Throughout this proof, all tensor norms
and inner products are taken with respect to $g_M$. 

By the assumption, there exist a constant $c>0$ and a nonempty set $O$
such that either $O\Subset M^\circ$ is open and $V\ge c$ on $O$,
or $O\subset S$ is relatively open and $B\ge c$ on $O$.
Choose an exhaustion
$$
 P_1\subset P_2\subset\cdots\subset M
$$
by connected compact smooth domains, with
$P_j\subset\operatorname{int}_M P_{j+1}$, such that $P_1$
contains $\overline O$ and a fixed collar of $S$.

For each $j$, define
$$
 \mathcal A_j(v)
 =
 \int_{P_j}|dv+v\omega|^2
 +\int_{P_j}Vv^2
 +\int_S Bv^2,
 \qquad v\in H^1(P_j).
$$
We claim that there is a constant $C_j>0$ such that
\begin{equation}\label{eq:claim}
 \int_{P_j}v^2\le C_j\mathcal A_j(v)
 \qquad\text{for every }v\in H^1(P_j).
\end{equation}
Fix $j$. If this estimate fails, there is a sequence
$v_k\in H^1(P_j)$ satisfying
$$
 \int_{P_j}v_k^2=1,
 \qquad
 \mathcal A_j(v_k)\longrightarrow0.
$$
The estimate
$$
 \int_{P_j}|dv_k|^2
 \le
 2\mathcal A_j(v_k)
 +2\|\omega\|_{L^\infty(P_j)}^2\int_{P_j}v_k^2
$$
shows that $\{v_k\}$ is bounded in $H^1(P_j)$.
By the Rellich-Kondrachov compactness theorem and the compactness of the trace map
\cite[Chapter~5]{Evans}, after passing to a subsequence,
$v_k$ converges weakly in $H^1(P_j)$ and strongly in $L^2(P_j)$
to a function $v$, and its traces converge strongly in $L^2(S)$
to the trace of $v$. Consequently,
$$
 \int_{P_j}v^2=1,
$$
while $\mathcal A_j(v_k)\to0$ gives
\begin{equation}\label{eq:first-order-zero}
 dv=-v\omega
\end{equation}
and $v=0$ on $O$, initially in the trace sense if $O\subset S$.

In local coordinates, by \eqref{eq:first-order-zero} we know
$\partial_\alpha v=-v\omega_\alpha$.
Since $\omega$ is smooth, these identities give $H^\ell$ regularity
for every $\ell$ by induction on smaller interior charts and boundary half-charts. By the Sobolev embedding theorem, $v$ has a representative that is smooth
up to the boundary, and vanishes pointwise on $O$. Using
this representative, and fixing a point $x_0\in O$, since $v$ is smooth up to the boundary,
we have $v(x_0)=0$, also in the case $O\subset S$.
For arbitrary $x\in P_j$, by connectedness of $P_j$, there
exists a piecewise smooth path $\gamma:[0,1]\to P_j$ such that
$$
 \gamma(0)=x_0,\qquad \gamma(1)=x.
$$
On each smooth segment of $\gamma$, by
\eqref{eq:first-order-zero} we have
$$
 \frac{d}{dt}v(\gamma(t))
 =
 dv_{\gamma(t)}(\dot\gamma(t))
 =
 -\omega_{\gamma(t)}(\dot\gamma(t))\,v(\gamma(t)).
$$
Solving this scalar linear ODE on each segment and using
continuity at the joining points, we obtain
\[
 v(x)
 =
 v(x_0)\exp\left(
 -\int_0^1\omega_{\gamma(t)}(\dot\gamma(t))\,dt
 \right)
 =0.
\]
Since $x\in P_j$ was arbitrary, it follows that $v\equiv0$
on $P_j$, contradicting $\int_{P_j}v^2=1$.
This proves the claim \eqref{eq:claim}.

For every $v\in C_c^\infty(M)$, we have
\[
\begin{aligned}
 \int_{P_1}|dv|^2
 &\le
 2\int_{P_1}|dv+v\omega|^2
 +2\|\omega\|_{L^\infty(P_1)}^2\int_{P_1}v^2\\
 &\le
 2\left(1+C_1\|\omega\|_{L^\infty(P_1)}^2\right)
 \mathcal A(v).
\end{aligned}
\]
Together with $\int_{P_1}v^2\le C_1\mathcal A(v)$,
the trace theorem on $P_1$ therefore gives
\[
\begin{aligned}
 \int_Sv^2
 &\le
 C_{\mathrm{tr}}
 \left(\int_{P_1}v^2+\int_{P_1}|dv|^2\right)\\
 &\le
 C_{\mathrm{tr}}
 \left(C_1+2+2C_1\|\omega\|_{L^\infty(P_1)}^2\right)
 \mathcal A(v)
 =:C_S\mathcal A(v).
\end{aligned}
\]
Here $C_{\mathrm{tr}}>0$ is the trace constant for the fixed
domain $P_1$, so $C_S$ is independent of $v$.

Choose a nonnegative locally finite smooth partition of unity
$\{\chi_j\}_{j\ge1}$ on $M$ with $\supp\chi_j\subset P_j$, and set
$$
 \varpi
 =
 \frac12\sum_{j\ge1}\frac{2^{-j}\chi_j}{1+C_j},
 \qquad
 \varepsilon_0=\frac1{2C_S},
$$
where $C_j$ are constants from \eqref{eq:claim}.
Local finiteness ensures that $\varpi$ is smooth, and
$\sum_j\chi_j=1$ implies $\varpi>0$ on $M$.
Since $0\le\chi_j\le1$, for every $v\in C_c^\infty(M)$ we have
$$
 \int_M\varpi v^2
 \le
 \frac12\sum_{j\ge1}
 \frac{2^{-j}}{1+C_j}\int_{P_j}v^2
 \le
 \frac12
 \left(\sum_{j\ge1}\frac{2^{-j}C_j}{1+C_j}\right)
 \mathcal A(v)
 \le\frac12\mathcal A(v).
$$
Moreover,
$$
 \varepsilon_0\int_Sv^2\le\frac12\mathcal A(v).
$$
Adding these inequalities proves \eqref{eq:anchored-form}.
\end{proof}

In the equality case, the graph equations yield $\Hess_g W=Wg$ on $X$. The following lemma establishes the corresponding rigidity result by integrating this equation along geodesics starting at a minimum point of $W$.

\begin{lemma}\label{lem:radial-closure}
Let $(X,g)$ be a smooth, connected, compact  Riemannian manifold of dimension $n\ge 3$ with nonempty smooth boundary $S$. Suppose that $W>0$ is smooth
up to $S$ and, for some constant $W_0>1$,
\begin{equation}\label{eq:radial-hypotheses}
 \Hess_g W=Wg,\qquad W^2-|dW|_g^2=1,\qquad
 W|_S=W_0,\qquad \partial_\nu W>0,
\end{equation}
where $\nu$ is the outward unit normal along $S$.
Then $(X,g)$ is isometric to the closed hyperbolic ball of radius
$\arcosh W_0$. Moreover, $W$ has a unique minimum point $o\in X^\circ$,
which corresponds to the center of the hyperbolic ball, and
$$
 W(x)=\cosh d_g(o,x)
 \qquad\text{for every }x\in X.
$$
\end{lemma}

\begin{proof}
Set $R=\arcosh W_0$. Since $W>0$ and
$W^2-|dW|_g^2=1$, we have $W\ge1$.
Moreover, taking the trace of $\Hess_g W=Wg$ gives
$\Delta_g W=nW>0$. The maximum principle and the boundary condition
therefore imply $W<W_0$ on $X^\circ$.
Thus $W$ attains its minimum at some point $o\in X^\circ$.
At this point $dW=0$, and hence $W(o)=1$.

For each $e\in\SSS(T_oX)$, let $\gamma_e$ be the unit speed
geodesic starting at $o$ with initial velocity $e$.
As long as $\gamma_e$ remains in $X^\circ$, the Hessian
equation gives
\begin{equation}\label{eq:radial-ode}
 (W\circ\gamma_e)''=W\circ\gamma_e,
 \qquad
 (W\circ\gamma_e)(0)=1,
 \qquad
 (W\circ\gamma_e)'(0)=0.
\end{equation}
Consequently,
$$
 W(\gamma_e(t))=\cosh t.
$$

Let $[0,T_e)$ be the maximal interval on which $\gamma_e$
is defined such that $\gamma_e([0, T_e))\subset X^\circ$.
Since $\cosh t<W_0=\cosh R$ for $t<T_e$, we have $T_e\le R$.
The unit speed condition gives
$$
 d_g(\gamma_e(s),\gamma_e(t))\le |s-t|.
$$
Hence compactness of $X$ implies that $\gamma_e(t)$ converges
to a point $p_e\in X$ as $t\uparrow T_e$.
If $p_e\in X^\circ$, the standard extension theorem for solutions of the geodesic
equation would extend $\gamma_e$ beyond $T_e$ inside $X^\circ$. This contradicts the maximality of $T_e$. Thus $p_e\in S$, and by continuity we have
$$
 \cosh T_e=W(p_e)=W_0=\cosh R.
$$
It follows that $T_e=R$.
Setting $\gamma_e(R)=p_e$, we obtain
$$
 \gamma_e([0,R))\subset X^\circ,
 \qquad
 \gamma_e(R)\in S.
$$
A smooth extension of $g$ across $S$ provides smooth continuation
through this endpoint and smooth dependence on the initial
direction $e$ up to $t=R$.

The exponential map therefore defines a smooth map
$$
 E:\overline{B_R(0)}\subset T_oX\longrightarrow X,
 \qquad E(z)=\exp_o(z),
$$
where $B_R(0)$ is the ball of radius $R$ centered at the origin in $T_oX$ equipped with the inner product $g_o$.
It satisfies
$$
 E(te)=\gamma_e(t),\qquad
 E(B_R(0))\subset X^\circ,\qquad
 E(\partial B_R(0))\subset S.
$$

Along each radial geodesic, we have
\begin{equation}\label{eq:radial-gradient}
 \nabla_g W(\gamma_e(t))=\sinh t\,\dot\gamma_e(t).
\end{equation}
Indeed, writing $\nabla_t$ for covariant differentiation along
$\gamma_e$, the Hessian and geodesic equations give
$$
 \nabla_t\bigl(\nabla_g W-\sinh t\,\dot\gamma_e\bigr)
 =(W-\cosh t)\dot\gamma_e=0,\qquad\text{ whenever } \gamma_e\subset X^\circ.
$$
The difference is therefore parallel along $\gamma_e$.
Since the difference vanishes at $t=0$, \eqref{eq:radial-gradient} holds
for $0\le t<R$ and extends to $t=R$ by continuity.

Fix $e\in\SSS(T_oX)$ and $v\in T_e\SSS(T_oX)$.
Choose a smooth curve $e(s)$ in $\SSS(T_oX)$ with
$e(0)=e$ and $e'(0)=v$, and consider the geodesic variation
$\mathcal I(s,t)=E(te(s))$.
Its variational Jacobi field is
$$
 J_v(t)
 =\partial_s\mathcal I(0,t)
 =(dE)_{te}(tv),
$$
with initial conditions
$$
 J_v(0)=0,\qquad \nabla_tJ_v(0)=v.
$$
Differentiating \eqref{eq:radial-gradient} along this variation,
and using
$\nabla_s\partial_t\mathcal I=\nabla_t\partial_s\mathcal I$, we have
$$
 \cosh t\,J_v
 =\nabla_{J_v}\nabla_g W
 =\sinh t\,\nabla_tJ_v
 \qquad\text{for }t>0.
$$
Hence $\nabla_t(J_v/\sinh t)=0$, and the initial conditions yield
\begin{equation}\label{eq:radial-jacobi}
 J_v(t)=\sinh t\,P_tv,
\end{equation}
where $P_t:T_oX\to T_{\gamma_e(t)}X$ denotes parallel transport
along $\gamma_e$.

We now compute the pullback metric.
For $t>0$, the differential of the polar coordinate map
$(t,e)\mapsto te$ sends a tangent vector $(a,v)$ to $ae+tv$.
Since $\dot\gamma_e(t)=P_te$, we have
$$
 (dE)_{te}(ae+tv)
 =a\dot\gamma_e(t)+J_v(t)
 =aP_te+\sinh t\,P_tv.
$$
Note parallel transport preserves inner products, and
$v,w\in T_e\SSS(T_oX)$ are perpendicular to $e$.
Thus, for $a,b\in\mathbb R$,
$$
 (E^*g)_{te}(ae+tv,be+tw)
 =ab+\sinh^2t\,g_o(v,w).
$$
Consequently, in polar coordinates $z=te$ on the punctured ball,
\begin{equation}\label{eq:radial-metric}
 E^*g
 =dt^2+\sinh^2t\,g_{\SSS(T_oX)}
 \qquad\text{for }0<t\le R,
\end{equation}
where $g_{\SSS(T_oX)}$ is the round metric induced by $g_o$.

The right hand side of \eqref{eq:radial-metric} is positive
definite for $t>0$, so $dE$ is invertible away from the origin.
At the origin, $dE_0=\mathrm{id}_{T_oX}$.
On $S$, the constant boundary value of $W$ gives
$\nabla_g W=(\partial_\nu W)\nu$.
Together with \eqref{eq:radial-gradient} and
$\partial_\nu W>0$, this implies
$$
 \dot\gamma_e(R)=\nu_{\gamma_e(R)}.
$$
Thus $dE$ maps the outward radial direction along
$\partial B_R(0)$ to the outward normal of $S$.
Since $E$ maps boundary to boundary and interior to interior,
the inverse function theorem in boundary charts shows that
$E$ is a local diffeomorphism at boundary points as well.
Therefore $E$ is a local diffeomorphism of manifolds with boundary.

To prove injectivity, suppose that
$$
 E(te_1)=E(se_2)=x,
 \qquad 0\le t,s\le R.
$$
Since $W(E(te_1))=\cosh t$ and $W(E(se_2))=\cosh s$,
the equality $E(te_1)=E(se_2)$ implies $t=s$.
If $t=s=0$, both points in the domain are the origin.
If $t=s>0$, equation \eqref{eq:radial-gradient} gives
$$
 \dot\gamma_{e_1}(t)
 =\frac{\nabla_g W(x)}{\sinh t}
 =\dot\gamma_{e_2}(t).
$$
Uniqueness for the geodesic equation, applied backwards from
this common endpoint and velocity, yields $e_1=e_2$.
Thus $E$ is injective.

The image of $E$ is closed in $X$ because its domain is compact,
and it is relatively open because $E$ is a local diffeomorphism,
including at boundary points.
Since $X$ is connected, $E$ is surjective.
It follows that $E$ is a diffeomorphism.
Equation \eqref{eq:radial-metric} then identifies $(X,g)$
with the closed hyperbolic ball of radius $R$.

Under this isometry, $o$ corresponds to the center of the
hyperbolic ball, and the radial parameter equals the distance
from $o$. Hence
$$
 W(x)=\cosh d_g(o,x)
 \qquad\text{for every }x\in X.
$$
In particular, $W(o)=1$ and $W(x)>1$ for $x\ne o$, so $o$
is the unique minimum point of $W$.
\end{proof}

\section{Proof of the main theorem}\label{sec:mainproofs}

We now combine the quadratic form obstruction of
Section \ref{sec:obstruction} with the Jang graph identities of Section \ref{sec:jang} to prove Theorem \ref{thm:hyperbolic}. The rigidity argument reduces to the Hessian equation treated
in Lemma \ref{lem:radial-closure}. We then prove Corollary \ref{cor:hyper-map} using \eqref{eq:degree-volume} and the radial extension of round sphere isometries to the hyperbolic ball.

\begin{proof}[Proof of Theorem \ref{thm:hyperbolic}]
Set
$$
 r=\sinh R,\qquad D_0=\frac{n-1}{r}.
$$
We will prove
\begin{equation}\label{eq:hyper-threshold}
 \rad(S,g_S)\ge r
 \quad\Longrightarrow\quad
 (X,g)\text{ is isometric to }\overline{B_{\HH^n}(R)}.
\end{equation}
The round sphere of radius $r$ has hyperspherical radius $r$. Indeed, the identity map gives the lower bound, and
\eqref{eq:degree-volume} gives the upper bound.
Thus \eqref{eq:hyper-threshold} implies both the comparison
and the characterization of equality.

Assume that $\rad(S,g_S)\ge r$.
Since $H_S\ge(n-1)\coth R>n-1$, Lemma~\ref{lem:jang}
provides a connected relatively open set $\Omega\subset X$
containing a collar of $S$ and a smooth function
$u:\Omega\to[0,\infty)$ with $u|_S=0$.
Its graph
$$
 \Gamma=\{(x,u(x)):x\in\Omega\}
$$
is properly embedded in $X\times[0,\infty)$ and complete with
respect to its induced metric.
We identify $\Gamma$ with $\Omega$ by the graph parametrization
and use the notation of Section \ref{sec:jang}:
$$
 W=\sqrt{1+|du|_g^2},\qquad
 \bar g=g+du^2,\qquad
 A_\Gamma=\frac{\Hess_g u}{W},
$$
with $q$ as in \eqref{eq:jang-tensors}.
The induced metric on $\partial\Gamma=S$ is $g_S$.

Define
$$
 V=\frac12\left(
       \Sc_g+n(n-1)+|A_\Gamma+g|_{\bar g}^2
     \right),
 \qquad
 B=H_{\partial\Gamma}-q(\eta_\Gamma)-D_0.
$$
Here $\Sc_g$ on $\Gamma$ denotes $\Sc_g\circ\pi$,
where $\pi:\Gamma\to\Omega$ is the graph projection.
The scalar curvature assumption and Lemma \ref{lem:graph-form}
imply $V\ge0$ on $\Gamma$ and $B\ge0$ on $S$.
Moreover, \eqref{eq:graph-square} becomes
$$
 \mathcal E_{D_0}(v)
 =
 \int_\Gamma|dv+vq|_{\bar g}^2
 +\int_\Gamma Vv^2
 +\int_S Bv^2
 \qquad
 \text{for every }v\in C_c^\infty(\Gamma).
$$

We claim that $V$ and $B$ vanish identically.
Assume, to the contrary, that at least one of $V$ and $B$
is not identically zero.
Applying Lemma \ref{lem:anchor} with $\omega=q$ gives
a smooth function $\varpi>0$ on $\Gamma$ and a constant
$\varepsilon_0>0$ such that
\begin{equation}\label{eq:epsD}
 \mathcal E_{D_0}(v)
 \ge
 \int_\Gamma\varpi v^2+\varepsilon_0\int_Sv^2
 \qquad
 \text{for every }v\in C_c^\infty(\Gamma).
\end{equation}
With $\varepsilon_0$ as above, choose $r'\in(0,r)$ sufficiently close
to $r$ that
\begin{equation}\label{eq:boundary-map-choice}
 0<\xi-D_0<\frac{\varepsilon_0}{2},\qquad \text{ where }\xi=\frac{n-1}{r'}.
\end{equation}
Since $r'<r\le\rad(S,g_S)$, definition~\eqref{eq:radius}
provides a radius $r_1>r'$ and a smooth map
$F_1:(S,g_S)\to\SSS^{n-1}(r_1)$ with
$\deg F_1\ne0$ and $\Lip F_1\le1$.
Composing $F_1$ with the radial dilation from $\SSS^{n-1}(r_1)$ onto $\SSS^{n-1}(r')$, we obtain the smooth map $F:
 (S,g_S)\rightarrow\SSS^{n-1}(r')$ defined by
$$
 F=\frac{r'}{r_1}F_1,
$$
which satisfies
$$
 \deg F=\deg F_1\ne0,
 \qquad
 \Lip F\le\frac{r'}{r_1}<1.
$$
The normalized map
$$
 \Phi=\frac{F}{r'}:S\longrightarrow\SSS^{n-1}
$$
has nonzero degree and satisfies $\Lip\Phi\le1/r'$.
Thus every singular value of $d\Phi$ is at most $1/r'$, and
$$
 D_\Phi=\|d\Phi\|_{\tr}\le\frac{n-1}{r'}=\xi.
$$
Using the definition of $\mathcal E_\xi$, \eqref{eq:epsD} and
\eqref{eq:boundary-map-choice}, we obtain
$$
 \begin{aligned}
 \mathcal E_\xi(v)
 &=
 \mathcal E_{D_0}(v)-(\xi-D_0)\int_Sv^2\\
 &\ge
 \int_\Gamma\varpi v^2
 +\bigl(\varepsilon_0-(\xi-D_0)\bigr)\int_Sv^2\\
 &\ge
 \int_\Gamma\varpi v^2
 +\frac{\varepsilon_0}{2}\int_Sv^2
 \end{aligned}
$$
for every $v\in C_c^\infty(\Gamma)$.

The manifold $(\Gamma,\bar g)$ is connected, oriented, and
complete, with nonempty smooth compact boundary $S$.
Since $\deg\Phi\ne0$ and $\xi\ge D_\Phi$, the preceding
inequality contradicts Corollary~\ref{cor:positive-form}.
This proves that $V\equiv0$ and $B\equiv0$.
Then the two nonnegative terms defining $V$
vanish separately. Using the definition of $q$ and  $B=0$, we obtain
\begin{equation}\label{eq:vanishing}
 \Sc_g=-n(n-1),\qquad
 A_\Gamma=-g,\qquad
 q=0
 \quad\text{on }\Gamma,
 \qquad
 B=0\quad\text{on }S.
\end{equation}

We now recover the geometry of $X$.
Since $A_\Gamma=\Hess_g u/W$, we have
$$
 \Hess_g u=-Wg
 \qquad\text{on }\Omega.
$$
Differentiating $W^2=1+|du|_g^2$ yields
$$
 WdW
 =\Hess_g u(\nabla_g u,\cdot)
 =-Wdu.
$$
Hence $dW=-du$, and connectedness of $\Omega$ gives a constant
$W_0$ such that
\begin{equation}\label{eq:W-plus-u}
 W+u=W_0
 \qquad\text{on }\Omega.
\end{equation}
In particular, $W|_S=W_0$ because $u|_S=0$.

Since $u\ge0$ and $W\ge1$, equation \eqref{eq:W-plus-u} gives
$$
 0\le u\le W_0-1.
$$
Thus $\Gamma$ is contained in the compact slab
$X\times[0,W_0-1]$.
The inclusion of $\Gamma$ into $X\times[0,\infty)$ is proper,
so $\Gamma$ itself is compact.
Its projection $\Omega$ is therefore compact and hence closed
in $X$.
Since $\Omega$ is also nonempty and relatively open,
connectedness of $X$ implies $\Omega=X$.

The functions $u$ and $W$ are consequently smooth on all of $X$,
up to $S$. From $dW=-du$ and $\Hess_g u=-Wg$, we obtain
\begin{equation}\label{eq:hyper-hessian}
 \Hess_g W=Wg,\qquad
 W^2-|dW|_g^2=1,\qquad
 W|_S=W_0.
\end{equation}
Taking the trace of the equation for $u$ gives
$\Delta_g u=-nW<0$.
The strong maximum principle and the Hopf lemma,
together with $u|_S=0$, imply
$$
 u>0\quad\text{in }X^\circ,
 \qquad
 \partial_\nu u<0\quad\text{on }S,
$$
where $\nu$ is the outward unit normal.
It follows from $W+u=W_0$ and $W\ge1$ that $W_0>1$, and
$$
 \partial_\nu W=-\partial_\nu u>0.
$$
Since $W$ is constant on $S$, its tangential derivatives vanish.
The second identity in \eqref{eq:hyper-hessian} therefore gives
$$
 \partial_\nu W=\sqrt{W_0^2-1}
 \qquad\text{on }S.
$$

Let $A_S$ denote the second fundamental form of $S$ with respect
to $\nu$, with convention
$A_S(Y_1, Y_2)=g(\nabla_{Y_1}\nu, Y_2)$.
The constant boundary value of $W$ gives
$$
 W_0g_S
 =\Hess_g W|_{TS}
 =(\partial_\nu W)A_S.
$$
Taking the trace, we obtain
$$
 H_S=\frac{(n-1)W_0}{\sqrt{W_0^2-1}}.
$$
On the other hand, $q=B=0$ implies $H_{\partial\Gamma}=D_0$.
The boundary calculation in Lemma~\ref{lem:graph-form} gives
$H_{\partial\Gamma}=H_S/W_0$, and hence
$$
 \frac{n-1}{\sinh R}
 =D_0
 =\frac{H_S}{W_0}
 =\frac{n-1}{\sqrt{W_0^2-1}}.
$$
Since $W_0>1$, this yields $W_0=\cosh R$.

All hypotheses of Lemma~\ref{lem:radial-closure} now hold.
It follows that $(X,g)$ is isometric to the closed hyperbolic
ball of radius $\arcosh W_0=R$.
This proves \eqref{eq:hyper-threshold} and completes the proof.
\end{proof}

We now prove Corollary~\ref{cor:hyper-map} by combining
Theorem~\ref{thm:hyperbolic} with the degree formula and
a radial extension of the resulting boundary isometry.

\begin{proof}[Proof of Corollary~\ref{cor:hyper-map}]
Set $r=\sinh R$.
Since the given map $F$ has nonzero degree and $\Lip F\le1$,
the definition of the hyperspherical radius gives
$\rad(S,g_S)\ge r$.
Theorem \ref{thm:hyperbolic} then gives
$\rad(S,g_S)=r$ and an isometry
$$
 I:(X,g)\longrightarrow\overline{B_{\HH^n}(R)}.
$$
Let $I_S=I|_S$.
The boundary of the hyperbolic ball, with its induced metric,
is isometric to $\SSS^{n-1}(r)$. Hence
$$
 \Vol_{g_S}(S)=\Vol(\SSS^{n-1}(r)).
$$

We first prove that $F$ is an isometry.
Let $\sigma_1(x),\ldots,\sigma_{n-1}(x)$ be the singular values
of $dF_x$ with respect to the given metrics.
Since $\Lip F\le1$, we have $0\le\sigma_i(x)\le1$, and therefore
$$
 J_F(x):=|\operatorname{Jac}F(x)|
 =\prod_{i=1}^{n-1}\sigma_i(x)\le1.
$$
The degree formula gives
\begin{align*}
 |\deg F|\,\Vol(\SSS^{n-1}(r))
 &\le \int_S J_F\,d\mu_{g_S}\\
 &\le \Vol_{g_S}(S)
 =\Vol(\SSS^{n-1}(r)).
\end{align*}
Since $|\deg F|\ge1$, equality must hold in both
inequalities above. In particular,
$$
 |\deg F|=1,
 \qquad
 \int_S(1-J_F)\,d\mu_{g_S}=0.
$$
The function $1-J_F$ is continuous and nonnegative, so
$J_F\equiv1$.
As every singular value is at most one, this forces
$\sigma_i(x)=1$ for every $i$ and every $x\in S$.
Thus $dF_x$ preserves inner products at every point, and
$F$ is a local isometry.

To see that $F$ is bijective, note that its signed
Jacobian is nowhere zero and has constant sign because
$S$ is connected.
Every point of the target is a regular value, and its
preimage is finite by compactness of $S$.
The degree formula for regular values therefore gives
$$
 \#F^{-1}(y)=|\deg F|=1
 \qquad
 \text{for every }y\in\SSS^{n-1}(r).
$$
Consequently, $F$ is a bijective local isometry and hence
a global isometry.

We now construct the required extension.
Using the radial isometry
$\iota:\SSS^{n-1}(r)\to\partial B_{\HH^n}(R)$
from the statement, define
$$
 A:=\iota\circ F\circ I_S^{-1}
 :
 \partial B_{\HH^n}(R)
 \longrightarrow
 \partial B_{\HH^n}(R).
$$
Thus, $A$ is an isometry of $\partial B_{\HH^n}(R)$
with respect to the metric induced by the hyperbolic
metric.
Let $o_{\HH}$ be the center of the hyperbolic ball, and denote
the radial parametrization of its boundary by
$$
 \Psi(e)=\exp_{o_{\HH}}(Re),
 \qquad e\in\SSS(T_{o_{\HH}}\HH^n).
$$
Note 
$\Psi^*g_{\partial B_{\HH^n}(R)}
 =\sinh^2R\,g_{\SSS(T_{o_{\HH}}\HH^n)}
 =r^2g_{\SSS(T_{o_{\HH}}\HH^n)}.$
Set
$$
 a=\Psi^{-1}\circ A\circ\Psi.
$$
Since $\Psi\circ a=A\circ\Psi$ and $A$ is an isometry,
$$
 r^2a^*g_{\SSS(T_{o_{\HH}}\HH^n)}
 =\Psi^*\bigl(A^*g_{\partial B_{\HH^n}(R)}\bigr)=\Psi^*g_{\partial B_{\HH^n}(R)}
 =r^2g_{\SSS(T_{o_{\HH}}\HH^n)}.
$$
Dividing by $r^2$, we conclude that $a$ is an isometry
of the round sphere.

Writing $d_{\SSS}$ for the distance on this unit sphere,
the spherical distance formula gives, for any unit vectors
$e$ and $\tilde{e}$,
\begin{equation}\label{eq:innerproa}
 \langle a(e),a(\tilde{e})\rangle
 =\cos d_{\SSS}(a(e),a(\tilde{e}))
 =\cos d_{\SSS}(e,\tilde{e})
 =\langle e,\tilde{e}\rangle.
\end{equation}
Here $\langle\cdot, \,\cdot\rangle$ denotes the inner product on $T_{o_{\HH}}\HH^n$ induced by the hyperbolic metric $g_{\HH^n}$ at $o_{\HH}$.
Choose an orthonormal basis $e_1,\ldots,e_n$ of
$T_{o_{\HH}}\HH^n$ and define the linear map
$$
 Q:T_{o_{\HH}}\HH^n\longrightarrow T_{o_{\HH}}\HH^n,
 \qquad Qe_i=a(e_i).
$$
\eqref{eq:innerproa} shows that
$a(e_1),\ldots,a(e_n)$ is an orthonormal basis, so $Q$ is
orthogonal. For every unit vector $e$, expansion in this
image basis gives
\begin{align*}
 a(e)
 =\sum_{i=1}^n
   \langle a(e),a(e_i)\rangle\,a(e_i)=\sum_{i=1}^n
   \langle e,e_i\rangle\,a(e_i)
 =Qe.
\end{align*}
Since $A\circ\Psi=\Psi\circ a$, it follows that
$$
 A\bigl(\exp_{o_{\HH}}(Re)\bigr)
 =\exp_{o_{\HH}}(RQe)
 \qquad
 \text{for every }e\in\SSS(T_{o_{\HH}}\HH^n).
$$

We extend $A$ to the closed hyperbolic ball by setting
$$
 \widetilde A\bigl(\exp_{o_{\HH}}z\bigr)
 =\exp_{o_{\HH}}(Qz),
 \qquad |z|\le R.
$$
Since $Q$ preserves norms, the image lies in the same closed
ball. The inverse is given by
$$
 \widetilde A^{-1}\bigl(\exp_{o_{\HH}}z\bigr)
 =\exp_{o_{\HH}}(Q^{-1}z),
 \qquad |z|\le R.
$$
In normal coordinates centered at $o_{\HH}$, the maps
$\tilde A$ and $\tilde A^{-1}$ are represented by the
linear maps $z\mapsto Qz$ and $z\mapsto Q^{-1}z$,
respectively.
Thus $\widetilde A$ is a smooth diffeomorphism of the closed
ball, including at the center.

For $0<t\le R$ and $e\in\SSS(T_{o_{\HH}}\HH^n)$,
the identity $Qe=a(e)$ gives
$$
 \widetilde A\bigl(\exp_{o_{\HH}}(te)\bigr)
 =\exp_{o_{\HH}}\bigl(ta(e)\bigr).
$$
Thus $\widetilde A$ preserves the radial parameter and
transforms the angular variable by $a$.
Since $a$ preserves the round metric, the hyperbolic
metric $g_{\HH^n}$ satisfies, for $0<t\le R$,
\begin{align*}
 \widetilde A^*g_{\HH^n}
 &=dt^2+\sinh^2t\,a^*g_{\SSS(T_{o_{\HH}}\HH^n)}\\
 &=dt^2+\sinh^2t\,g_{\SSS(T_{o_{\HH}}\HH^n)}
 =g_{\HH^n}.
\end{align*}
Both $\widetilde A^*g_{\HH^n}$ and $g_{\HH^n}$ are smooth
at the center, so we also have $\widetilde A^*g_{\HH^n}=g_{\HH^n}$ at $o_{\HH}$ by continuity.
Hence $\widetilde A$ is an isometry of the closed hyperbolic ball. It is clear that
$$
 \widetilde A|_{\partial B_{\HH^n}(R)}=A.
$$

Finally, define
$$
 \overline F=\widetilde A\circ I.
$$
As a composition of isometries, $\overline F$ is an isometry
from $(X,g)$ onto $\overline{B_{\HH^n}(R)}$.
Moreover,
$$
 \overline F|_S
 =A\circ I_S
 =\iota\circ F
$$
This completes the proof.
\end{proof}

Corollary \ref{cor:euclidean-radius} follows from Theorem \ref{thm:hyperbolic} and \cite[Theorem 1.4]{WWXZ}, with an additional regularity argument in the equality case.
Indeed, the hyperspherical radius is a priori attained only by a Lipschitz map, whose smoothness must be established before applying \cite[Theorem 1.4]{WWXZ}.
\begin{proof}[Proof of Corollary~\ref{cor:euclidean-radius}]
Applying the sharp mean curvature estimate
\eqref{eq:gromov-hyperbolic}, which follows from
Theorem \ref{thm:hyperbolic}, to $\varepsilon g$ gives
$$
 \frac{n-1}{r}\le\inf_S H_S
 \le(n-1)\sqrt{\epsilon+\rad(S,g_S)^{-2}}
 \qquad\text{for every }\varepsilon>0.
$$
Indeed, $\Sc_{\varepsilon g}=\varepsilon^{-1}\Sc_g\ge0$, while
$H_S^{\varepsilon g}=\varepsilon^{-1/2}H_S$ and
$\rad(S,\varepsilon g_S)=\sqrt{\varepsilon}\,\rad(S,g_S)$.
Letting $\varepsilon\downarrow0$ proves $\rad(S,g_S)\le r$.

Suppose now that $\rad(S,g_S)=r$. Replacing $g$ by $r^{-2}g$, we may assume that $r=1$. We first extend the quadratic form obstruction of Corollary \ref{cor:positive-form} to Lipschitz boundary maps. For a Lipschitz map $\Psi:S\to\SSS^{n-1}$, let $D_\Psi=\|d\Psi\|_{\tr}$ almost everywhere and define
$$
\begin{aligned}
 Q_\Psi(v)&=
 \int_X\left(|dv|_g^2+\frac12\Sc_gv^2\right)
 +\int_S(H_S-D_\Psi)v^2,\\
 \lambda(\Psi)&=
 \inf_{0\ne v\in H^1(X)}\frac{Q_\Psi(v)}{\|v\|_{L^2(X)}^2}.
\end{aligned}
$$
Boundary values of $v$ are understood as traces. The trace interpolation inequality bounds $Q_\Psi$ below by a multiple of $-\|v\|_{L^2(X)}^2$. By the direct method in the calculus of variations, using the compactness of the Sobolev embedding and the trace map, the infimum is attained. We claim that
\begin{equation}\label{eq:euclidean-lipschitz-obstruction}
 \lambda(\Psi)\le0\qquad\text{whenever }\deg\Psi\ne0.
\end{equation}
If $\lambda(\Psi)>0$, combining the positive spectral lower bound with the trace interpolation inequality, we have
$Q_\Psi(v)\ge C_1\|v\|_{H^1(X)}^2$ for some $C_1>0$.
Fix $p>n-1$. Since $\Psi$ is Lipschitz, there exist smooth $\R^n$-valued maps $\widetilde\Psi_j$ converging to $\Psi$ uniformly and strongly in $W^{1,p}(S,\R^n)$.
For sufficiently large $j$, we have
$|\widetilde\Psi_j|\ge1/2$, and hence the maps
$$
 \Psi_j=\frac{\widetilde\Psi_j}
 {|\widetilde\Psi_j|}
$$
are smooth and take values in $\SSS^{n-1}$. Since the radial projection is smooth in a neighborhood of $\SSS^{n-1}$, the uniform and strong $W^{1,p}$ convergence of $\widetilde\Psi_j$ imply that $\Psi_j\to\Psi$ uniformly on $S$ and strongly in $W^{1,p}(S,\R^n)$.
Moreover, uniform convergence gives a homotopy between
$\Psi_j$ and $\Psi$, so $\deg\Psi_j=\deg\Psi$ for all sufficiently large $j$. Since the trace norm satisfies
$$
 \bigl|\|A\|_{\tr}-\|B\|_{\tr}\bigr|
 \le\|A-B\|_{\tr}
 \le\sqrt{n-1}|A-B|_{\mathrm{HS}},
$$
where $|\cdot|_{\mathrm{HS}}$ denotes the Hilbert--Schmidt norm, the strong convergence of $d\Psi_j$ to $d\Psi$ in $L^p(S)$ implies that $D_{\Psi_j}\to D_\Psi$ strongly in $L^p(S)$.
Since $p>n-1$, the trace theorem and the Sobolev embedding on $S$ yield that
$$
 \|v\|_{L^{2p/(p-1)}(S)}
 \le C\|v\|_{H^1(X)}.
$$
Consequently, by the H\"older inequality we obtain
$$
\begin{aligned}
 |Q_{\Psi_j}(v)-Q_\Psi(v)|
 &\le
 \|D_{\Psi_j}-D_\Psi\|_{L^p(S)}
 \|v\|_{L^{2p/(p-1)}(S)}^2\\
 &\le
 C\|D_{\Psi_j}-D_\Psi\|_{L^p(S)}
 \|v\|_{H^1(X)}^2.
\end{aligned}
$$
Choose $\xi_j\in C^\infty(S)$ with $\xi_j\ge D_{\Psi_j}$
and $\|\xi_j-D_{\Psi_j}\|_\infty\to0$.
Since $Q_\Psi$ is coercive and
$D_{\Psi_j}$ converges strongly to $D_\Psi$ in $L^p(S)$, the trace
theorem implies that, for all sufficiently large $j$,
$$
 \int_X\left(|dv|_g^2+\frac12\Sc_gv^2\right)
 +\int_S(H_S-\xi_j)v^2
 \ge C_2\|v\|_{H^1(X)}^2
$$
for some $C_2>0$. Applying the trace theorem once more, we obtain constants $a>0$ and $b>0$ such that
$$
 \int_X\left(|dv|_g^2+\frac12\Sc_gv^2\right)
 +\int_S(H_S-\xi_j)v^2
 \ge a\int_Xv^2+b\int_Sv^2
$$
for every $v\in C^\infty(X)$, without imposing any boundary
condition on $v$.
This contradicts Corollary \ref{cor:positive-form} and hence confirms the claim
\eqref{eq:euclidean-lipschitz-obstruction}.

Since $\rad(S,g_S)=1$, there are smooth maps
$F_j:S\to\SSS^{n-1}$ of nonzero degree with
$\Lip F_j\le1+1/j$. By the Arzel\`a--Ascoli theorem, a subsequence converges uniformly to a map $F:S\to\SSS^{n-1}$ with $\Lip F\le1$. For large $j$, the maps $F_j$ and $F$ are homotopic, so $\deg F\ne0$.
The inequalities $\Sc_g\ge0$ and $H_S\ge n-1\ge D_F$ imply $Q_F\ge0$. Thus by
\eqref{eq:euclidean-lipschitz-obstruction} we have $\lambda(F)=0$.
Since $Q_F$ is a sum of nonnegative terms and
$\lambda(F)=0$, a nonzero minimizer $u$ satisfies $\int_X|du|_g^2=0$. As $X$ is connected, $u$ is a nonzero constant. Vanishing of the other nonnegative
terms gives
\begin{equation}\label{eq:euclidean-saturation}
\begin{aligned}
 \Sc_g=0\quad\text{in }X,\qquad H_S=n-1\quad\text{on }S,\qquad
 D_F=n-1\quad\text{a.e. on }S.
\end{aligned}
\end{equation}
The $n-1$ singular values of $dF$ are at most one and have sum $n-1$. Hence $(dF)^*dF=I$ almost everywhere.

We next prove that $F$ is smooth. For any
$\zeta\in C^\infty(S,\R^n)$, set
$$
 F_t=\frac{F+t\zeta}{|F+t\zeta|}.
$$
For small positive and negative $t$, these are Lipschitz maps
homotopic to $F$. Since $(dF)^*dF=I$ almost everywhere,
expanding $D_{F_t}=\tr\bigl(((dF_t)^*dF_t)^{1/2}\bigr)$ gives
\begin{equation}\label{eq:euclidean-trace-variation}
 D_{F_t}=n-1+t b_\zeta+O(t^2)\quad\text{in }L^\infty(S),\qquad
 b_\zeta=\langle dF,d\zeta\rangle-(n-1)\langle F,\zeta\rangle.
\end{equation}
By \eqref{eq:euclidean-saturation}, $Q_F(v)=\int_X|dv|_g^2$. Let $u_0=\Vol_g(X)^{-1/2}$, and for each sufficiently small $t$, let $u_t$ be a nonnegative first eigenfunction associated with the quadratic form $Q_{F_t}$ and eigenvalue $\lambda(F_t)$, normalized by $\int_Xu_t^2=1$.
Testing the Rayleigh quotient with $u_0$ and using the trace interpolation inequality give
$$
 \lambda(F_t)=O(|t|),
 \qquad
 \int_X|du_t|_g^2=O(|t|).
$$
Connectedness and the Poincar\'e inequality, together with the normalization and nonnegativity of $u_t$, imply that
$u_t\to u_0$ strongly in $H^1(X)$. By the trace theorem, $u_t\to u_0$ strongly in $L^2(S)$, and hence
$$
 \int_S b_\zeta u_t^2
 =
 \frac1{\Vol_g(X)}\int_S b_\zeta+o(1).
$$
Evaluating the Rayleigh quotient at $u_0$ gives the upper bound, while using the minimizing eigenfunction $u_t$ and discarding the nonnegative Dirichlet energy gives the corresponding lower bound. Combining these estimates, we obtain
\begin{equation}\label{eq:euclidean-eigenvalue-variation}
 \lambda(F_t)
 =-\frac{t}{\Vol_g(X)}\int_S b_\zeta+o(|t|).
\end{equation}

Each $F_t$ has nonzero degree, so
\eqref{eq:euclidean-lipschitz-obstruction} gives $\lambda(F_t)\le0$ for both signs of $t$.
It follows from \eqref{eq:euclidean-eigenvalue-variation} that $\int_S b_\zeta=0$. We therefore have
$$
 \int_S\langle dF,d\zeta\rangle
 =(n-1)\int_S\langle F,\zeta\rangle.
$$
Since $\zeta\in C^\infty(S,\R^n)$ is arbitrary, it follows that $-\Delta_{g_S}F=(n-1)F$ weakly. Regularity theory for elliptic equation shows that $F$ is smooth, and $(dF)^*dF=I$ then holds everywhere.

We now have $\Sc_g=0$, $H_S=n-1=(n-1)\|dF\|_{\op}$, and $\deg F\ne0$. By \cite[Theorem 1.4]{WWXZ}, there exists $c>0$ such that $(X,cg)$ is isometric to the Euclidean unit ball. In particular, $H_S^{cg}=n-1$. Since mean curvature scales under constant rescaling of the metric as
$$
H_S^{cg}=c^{-1/2}H_S^g=c^{-1/2}(n-1),
$$
we conclude that $c=1$. Hence $(X,g)$ itself is isometric to the Euclidean unit ball.

Recall that we reduced to the case $r=1$ by replacing the original metric $g$ with $r^{-2}g$. Since the rescaled manifold is isometric to the closed Euclidean unit ball, rescaling back shows that $(X,g)$ is isometric to the closed Euclidean ball of
radius $r$. Conversely, for the Euclidean ball of radius $r$, the identity map on its boundary gives $\rad(S,g_S)\ge r$, while the comparison established at the beginning yields the reverse inequality. This completes the proof. 
\end{proof}

\bibliography{bib}
\bibliographystyle{plain}

\end{document}